\documentclass[12pt,a4paper,twoside,leqno]{amsart}

\usepackage[latin1]{inputenc}

\usepackage{amsmath}
\usepackage{amsfonts}
\usepackage{amssymb}
\usepackage{amstext}
\usepackage{amsbsy}
\usepackage{color}

\newtheorem{corollary}{\sc Corollary}[section]
\newtheorem{theorem}{\sc Theorem}[section]
\newtheorem{lemma}{\sc Lemma}[section]
\newtheorem{proposition}{\sc Proposition}[section]
\newtheorem{remark}{\sc Remark}[section]
\newtheorem{definition}{\sc Definition}[section]

\newcommand{\thmref}[1]{Theorem~\ref{#1}}

\newcommand{\lemref}[1]{Lemma~\ref{#1}}

\newcommand{\propref}[1]{Proposition~\ref{#1}}
\newcommand{\defref}[1]{Definition~\ref{#1}}

\def\Hip{\mathbb H}

\let\hip=\Hip

\def\Na{\mathbb N}

\newcommand{\R}{\mathbb R}
\newcommand{\M}{\mathbb M}

\newcommand{\wt}{\widetilde}
\newcommand{\wh}{\widehat}

\def\Gai{\Ga_infty}

\let\h=\hip

\let\re=\R

\def\Mc{\mathcal M}

\let\Pc =\P

\def\Sc{\mathcal S}

\def\Nc{\mathcal N}
\def\Dc{\mathcal D}
\def\Qc{\mathcal Q}

\def\Om{\Omega}

\def \la{\lambda}
\def \a{\alpha}
\def \om{\omega}
\def \ga{\gamma}

\def \a{\alpha}

\def \ga{\gamma}
\def \Ga{\Gamma}

\def \e{\varepsilon}

\def\Gai{\Ga_\infty}

\def \ep{\epsilon}
\def \la{\lambda}

\def\vphi{\varphi}

\def \om{\omega}
\def \Om{\Omega}

\def \dd   {\displaystyle}

\def\leqs{\leqslant}
\def\geqs{\geqslant}

\DeclareMathOperator{\inter}{int}

\def\rmd{\mathop{\rm d\kern -1pt}\nolimits}
\def\rme{\mathop{\rm e\kern -1pt}\nolimits}

\def\e{\varepsilon}

\DeclareMathOperator{\diver}{div}
\DeclareMathOperator{\dist}{dist}

\def \noi {\noindent}

\def\bel{ \medskip
 \centerline{$ \ast \hbox to 1.0cm{}\ast \hbox to 1.0cm{}\ast $}
}

\def\overl{\overline}

\def\goto{\rightarrow}

\def\longerrightarrow{-\kern-5pt\longrightarrow}

\def\vphi{\varphi}
\def\star{\lower 1pt\hbox{*}}
\def \nulset {
\raise 1pt\hbox{ \hskip -3pt$\not$\kern -0.2pt \raise
.7pt\hbox{${\scriptstyle\bigcirc}$}}}

\def \ep{\epsilon}

\newcommand{\hd}{\mathbb{H}^2}

\newcommand{\hi}[1]{\mathbb{H}^#1}

\newcommand{\pain}{\partial_{\infty}}

\newcommand{\ov}[1]{\overline{#1}}

\let\leq=\leqslant
\let\geq=\geqslant

\let\ep=\e

\begin{document}

\title[minimal]
{Minimal Graphs in $ {\mathbb H}^n\times \mathbb R$ and
$\R^{n+1}$}

\author[  R. Sa Earp  and E. Toubiana]
{\scshape R. Sa Earp and E. Toubiana}

 \address{Departamento de Matem\'atica,
  Pontif\'\i cia Universidade Cat\'olica do Rio de Janeiro, Rio de Janeiro,  22453-900 RJ,
 Brazil }\email{earp@mat.puc-rio.br}

\address{Institut de Math{\'e}matiques de Jussieu, Universit{\'e} Paris
VII, Denis Diderot, Case 7012,
         2 Place Jussieu,
         75251 Paris Cedex 05, France}
\email{toubiana@math.jussieu.fr}
\thanks{The first author wish to thank {\em Laboratoire G\'eom\'etrie et
Dynamique de l'Institut de Math\'ematiques de Jussieu} for the
kind hospitality and support. The authors would like to thank
CNPq, FAPERJ (``Cientistas do Nosso Estado''), PRONEX of Brazil
and Accord Brasil-France, for partial financial support}

\date{\today}


 \maketitle

% \keywords{barrier}

 \begin{flushright}

 En l'honneur de Pierre Bérard and Sylvestre Gallot
 \end{flushright}

\bigskip

\begin{abstract}

\noi

We construct  geometric barriers for minimal graphs in
$\hip^n\times \R.$

We prove the existence and uniqueness of a solution of the
vertical minimal equation in the interior of a convex polyhedron
in $\hip^n$ extending
 continuously to the interior of each face, taking    infinite boundary data on one  face and  zero boundary
value data  on the other faces.

 In $\hip^n\times \R$, we solve the Dirichlet problem for the
vertical minimal equation in a  $C^0$   convex domain
$\Om\subset\hip^n$ taking arbitrarily continuous  finite boundary
and asymptotic boundary data.

 We prove the existence
of another Scherk type hypersurface, given by the solution of the
vertical minimal equation in the interior of  certain admissible
polyhedron taking
 alternatively infinite values $+\infty$ and $-\infty$ on adjacent
 faces of this polyhedron.

We establish analogous results for minimal graphs when the ambient
is the Euclidean space $\R^ {n+1}$.

\end{abstract}

\maketitle

%\textbf{Mathematics Subject Classification (2000):~} 53C42, 35J25

\medskip


% \vspace{2cm}

\section{Introduction }\label{intro}

In Euclidean space, H. Jenkins and J. Serrin \cite{JS2} showed
that in a bounded $C^2$ domain $D$ the Dirichlet problem for the
minimal equation in $D$ is solved for $C^2$ boundary data if and
only if the boundary is mean convex. The theorem also holds in the
case that the boundary data is $C^0$ (but the domain is still
$C^2$) by an approximation argument \cite[Theorem 16.8]{GT}. On
the other hand, the authors solved the Dirichlet problem in
$\hip^3$ for the vertical minimal surface equation over a $C^0$
convex domain $\Om$ in $\pain \hip^3,$ taking any prescribed
continuous boundary data on $\partial \Om$ ~\cite{asi}. There are
also in this context the general results proved by M. Anderson
~\cite{An1} and ~\cite{An2}.

In this paper we study the vertical minimal equation equation in
$\hip^n\times \R$ (\defref{vgr}) in the same spirit of our
previous work when $n=2$ \cite{SE-T}. In that  paper the authors
have given a full description of the minimal surfaces in $\hd
\times \R$ invariant by translations (cf \cite{Sa}). Afterwards,
inspired on this construction, P. B\'erard and the first author
\cite{B-SE} have given the minimal translation hypersurfaces in
$\hi n \times \R$ and they showed that the geometric behavior is
similar to the two dimensional case. There is also a one parameter
family of such hypersurfaces, denoted again by $M_d,\ d>0$. For
instance, $M_1$ is a vertical graph over an open

half-space of $\hip^n$ bounded by a  geodesic hyperplane $\Pi$,
taking infinite boundary value data on $\Pi$ and zero asymptotic
boundary value data. We show that the hypersurface $M_1$ provides
a barrier to the Dirichlet problem at any point of the asymptotic
boundary of $\Om.$ Moreover, we prove that the hypersurfaces $M_d$
($d<1$) give a barrier to the Dirichlet problem at any strictly
convex point of the finite boundary of $\Om.$

We prove the existence and the uniqueness of rotational Scherk
hypersurfaces in $\hip^ n\times \R$ and we prove that these
hypersurfaces give a barrier to the Dirichlet problem at any
convex point.

Given an admissible convex polyhedron (\defref{pol}), we prove the
existence and uniqueness of a solution of the vertical minimal
equation in $\inter (\Pc)$ extending continuously to the interior
of each face,  taking infinite boundary value on one face and zero
boundary value data on the other faces.
 We call these minimal hypersurfaces in
$\hip^n\times \R$ by {\em  first Scherk type $($minimal$)$
hypersurface}. The hypersurface $M_1$ above plays a crucial role
in the construction.

Using the rotational Scherk hypersurfaces as barriers, we solve
the Dirichlet problem for the minimal vertical equation in a
bounded $C^0$ convex domain $\Om\subset\hip^n$ taking arbitrarily
continuous boundary data. Furthermore, using the hypersurface
$M_1$ as well, we are able to solve the Dirichlet problem for the
minimal vertical equation in a $C^0$ convex domain
$\Om\subset\hip^n$ taking arbitrarily continuous data along the
finite and asymptotic boundary.

We prove the existence of another Scherk type hypersurface, that
we call {\em Scherk second type  hypersurfaces},  given by the
solution of the vertical minimal equation in the interior of a
certain   polyhedron taking alternatively infinite values
$+\infty$ and $-\infty$ on adjacent faces of this polyhedron.
Those polyhedra may be chosen convex or non convex.

We establish also that the  above results, except the statements
involving the asymptotic boundary, hold for minimal graphs in $\R^
n\times \R=\R^ {n+1}.$

\vskip2mm

 Given  a non convex admissible domain $\Om\subset\hip^n$
and given certain geometric conditions on the asymptotic boundary
data $\Gamma_\infty \subset \pain \hi n \times~\R$, we prove the
existence of a minimal graph in $\hip^n\times \R$ whose finite
boundary is $\partial \Om$ and whose asymptotic boundary data is
$\Ga_\infty$.

\vskip2mm

 A further interesting open problem  is to prove a
  ``Jenkins-Serrin'' type results in $\hip^n\times \R$.
  When $n=2$ this task was carried
 out, for instance,
 by B.  Nelli and H. Rosenberg ~\cite{NR} or by L. Mazet,
 M. M. Rodriguez and H. Rosenberg~\cite{MRR}. Recently, A. Coutant \cite{Co}, under the
supervision of F. Pacard, has obtained Scherk type hypersurfaces
in
 $\R^{n+1}$ using a different approach.

\vskip2mm

The knowledge of the $n$-dimensional hyperbolic geometry is
usefull in this paper. The reader is referred to \cite{cas}.

{ The authors are grateful to the referee for his valuable
observations. }
\section{minimal hypersurfaces invariant by hyperbolic
translations in $\hi n \times \R$}\label{S.Md}

We recall shortly the geometric description of the family $M_d$ of
translation hypersurfaces. First consider a fixed geodesic
hyperplane $\Pi$ of $\hi n$. Let $O\in \Pi$ be any fixed point and
let $\gamma \subset \hi n$ be the complete geodesic through $O$
orthogonal to $\Pi$.

For any $d>0$, the hypersurface $M_d$ is generated by a curve in
the vertical geodesic two-plane $\gamma \times \R$. The orbit of a
point of the generating curve at level $t$ is the equidistant
hypersurface of $\Pi$ in $\hi n \times \{ t\}$ passing through
this point.

 As we said in the introduction, for $d=1$, the hypersurface $M_1$ is a complete non entire
vertical graph over a half-space of $\hi n \times \{0\}$ bounded
by $\Pi$, taking infinite value data on $\Pi$ and zero asymptotic
boundary value data.

For any $d<1$, the hypersurface $M_d$ is an entire vertical graph.
For $d>1$, $M_d$ is a bi-graph over the exterior of an equidistant
hypersurface in $\hip^n=\hi n \times \{0\}$.

The generating curve of $M_d$  is given by the following explicit
form:
\begin{equation}\label{trl}
t= \lambda(\rho)=\int_a^\rho \frac{d}{\sqrt{\cosh^{2n-2} u
 -d^2}}du,\qquad \text{($a\geqs 0$)}
\end{equation}
where $\rho $ denotes the signed distance on $\gamma$ with respect
to the point $O$. More precisely: if $d>1$ then $a>0$ satisfies
$\cosh^{n-1}(a)=d$ and $\rho \geqs a$, if $d=1$ then $\rho \geqs
a>0$ and if $d<1$ then $a=0$ and $\rho \in \R$. Observe that { if}
$d<1$ then $\lambda$ is an odd function of $\rho \in \R$.

It can be proved in the same way as in Proposition 2.1 of
\cite{SE-T} that for any $\rho >0$ we have
\begin{equation}\label{mdp}
\lambda (\rho)   \goto +\infty, \quad \textrm{if}\quad d\to 1\quad
(d\not=1). \qquad\text{($M_d$\,\textrm{-} Property)}
\end{equation}

\section{\sc Vertical minimal equation in $\hi n \times \R$ }

\begin{definition}[{\bf {\em Vertical graph}}]\label{vgr} ${}$
 Let $\Om\subset M$ be a domain in a  $n$-dimensional Riemannian
manifold $M$ and let $u:\Om\goto \R$ be a $C^2$  function on
$\Om$.
 A {\em
vertical graph} in the product space $M\times \R$ is a set
$G=\{(x,u(x))\mid x\in \Om\}.$ We call $u$ the height function.

Let $X$ be a vector field tangent to $M.$ We denote by
$\nabla_{\!\!M} u$ and by $\diver_{\!M} X$ the gradient of $u$ and
the divergence of $X$, respectively. We define
$W_{\!\!M}u:=\sqrt{1+\|\nabla_{\!\!M}u\|_M^2 }$.

\end{definition}

The following proposition is straightforward but we will write it
in a
 suitable form  to establish  the {\em reflection principle} we need.

\begin{proposition}[{\bf {\em Mean curvature equation in $M\times
\R$}}]\label{P.equation}
 Assume that the domain $\Om\subset M$ in  coordinates
 $(x_1,\ldots,x_n)$ is endowed by a conformal metric
 $\la^2(x_1,\ldots, x_n)\left(dx_1^2+\cdots+dx_n^2\right)$. Let $H$
 be the mean curvature of a  vertical graph $G$.  Then the height
function $u(x_1,\ldots,x_n)$   satisfies the following equation

\begin{equation}\label{meaequ}
\begin{split}
&nH=\diver_{\!\!M}\left(\frac{\nabla_{\!\!M}u}{W_{\!\!M}u}\right):= \Mc (u)\\
&=\sum\limits_{i=1}^n\frac{n\la_{x_i}u_{x_i}}{\la^3\sqrt{1+\la^{-2}\|\nabla
u\|^2_{ {\R^n}} }  } +\sum\limits_{i=1}^n\frac {\partial}{\partial
x_i}\left(\frac{\la^{\!-2}\, u_{x_i}}{\sqrt{1 +\la^{\!-2}\|\nabla u\|^2_{{\R^n}}}} \right)\\
&\text{\rm(Mean curvature equation)}\\
%&\text{Particularly, if $H=0$ we obtain the minimal equation}\\
%&\Mc(u):=\diver_{\!\!M}\left(\frac{\nabla_{\!\!M}u}{W_{\!\!M}u}\right)=0\qquad\text{\rm(Minimal equation)}\\
\end{split}
\end{equation}
\end{proposition}

\begin{proof}

Consider in the conformal coordinates $(x_1,\ldots,x_n)$ the frame
field $X_k=\frac{\partial}{\partial x_k}, k=1,\ldots, n.$ Then
 the upper unit normal field $N$ is given by
\begin{equation*}
\dd N=\frac{-\la^{-2}\sum\limits_{i=1}^ n u_{x_i}
\frac{\partial}{\partial x_i} +\frac{\partial}{\partial
t}}{\sqrt{1+\|\nabla u\|_{\!M}^ 2}}=
-\frac{\nabla_{\!\!M}u}{W_{\!\!M}u}
+\frac{1}{W_{\!\!M}u}\,\frac{\partial}{\partial t}.
\end{equation*}
 We call $\dd N^ {{ h}}~:=-\frac{\nabla_{\!\!M}u}{W_{\!\!M}u}$
the  horizontal component of $N$ (lifting of a vector field
tangent to $M$). Now using the properties of the Riemannian
connection, we infer that the divergence of $N$ in the ambient
space $\M\times \R$ is given by
$\dd\diver_{\!M\times\R}N=\diver_{\!\!M} N^{{ h}}.$ On the other
hand we have, $\dd\diver_{\!M\times\R}N=-nH,$ hence we obtain the
first equation in the statement of the proposition. Finally, the
second equation  follows from a simple derivation.
\end{proof}

From Proposition \ref{P.equation}, we deduce the {\em minimal

vertical equation} or simply {\em minimal equation} in
$\hip^n\times \R$ ($\Mc (u)=0$).  We observe that this equation
was  obtained in a more general setting by  Y.-L. Ou
\cite[Proposition 3.1]{O}.

\begin{corollary} [{\bf {\em Minimal  equation in $\hip^n\times\R$}}]

 Let us consider the upper half-space model of hyperbolic space:
 $\hip^n=\{(x_1,\ldots,x_n)\in\R^n \mid x_n>0\}$.
 If $H=0$, then the height function $u(x_1,\ldots,x_n)$ of a vertical minimal graph $G$
 satisfies the following equation

\begin{equation}\label{minequ}
\begin{split}
&\Mc(u):=\diver_{\R^n}\left(\frac{\nabla_{\R^ n} u}{\sqrt{1
+x_n^2(u_{x_1}^2 +\cdots+u_{x_n}^ 2)}}\right) \\
& \qquad \qquad \qquad \qquad  \qquad \qquad \qquad + \frac{(2-n)
u_{x_n}}{x_n\sqrt{(1 +x_n^2(u_{x_1}^2 +\cdots+u_{x_n}^ 2)}}=0, \\
&\text{or equivalently}\\
 &\sum\limits_{i=1}^n\left(1 +x_n^2(u_{x_1}^2
+\cdots+\wh{u_{x_i}^2}+\cdots +u_{x_n}^2)\right)u_{x_i x_i}\quad \\
&\qquad+\frac{(2-n)\!\left( 1 + x_n^2(u_{x_1}^2+\cdots
+u_{x_n}^2)\right) u_{x_n}}{x_n} -
2x_n^2\sum\limits_{i<k}u_{x_i}u_{x_k}u_{x_i
x_k}\\
&\hphantom{oooo00}-x_nu_{x_n}\left(u_{x_1}^2+\cdots
+u_{x_n}^2\right)=0\qquad \text{\rm (Minimal equation)}\\
\end{split}
\end{equation}

\end{corollary}

For example the hypersurfaces $M_d$, $d\in (0,1)$, are entire
vertical graphs whose the height function satisfies Equation
(\ref{minequ}). Other examples are provided by the half part of
the hypersurfaces $M_d$, $d>1$, and the half part of the
$n$-dimensional catenoid, \cite{B-SE} and \cite{SE-T}.

\smallskip

Now we state the classical maximum principle and uniqueness for
the equation (\ref{minequ}).

\begin{remark}[{\bf{\em Classical maximum principle}}]\label{clm}
{\em Let $\Omega \subset \hi n$ be a bounded domain and let $g_1,
g_2:\
\partial \Omega  \rightarrow \R$ be continuous functions
satisfying $g_1 \leqs g_2$. Let $u_i : \overline \Omega
\rightarrow \R$ be a continuous extension of $g_i$ on $\overline
\Omega$ satisfying the minimal equation $(\ref{minequ})$ on
$\Omega$, $i=1,2$, then we have $u_1\leqs u_2$ on $\Omega$.
Consequently, setting $g_1=g_2$, there is at most one continuous
extension of $g_1$ on $\overline \Omega$ satisfying the minimal
surface equation (\ref{minequ}) on $\Omega$.

\smallskip

We will need also a maximum principle involving the asymptotic
boundary.

Let $\Omega \subset \hi n$ be an unbounded domain and let $g_1,
g_2:\
\partial \Omega \cup \pain \Omega \rightarrow \R$ be bounded functions
 satisfying $g_1 \leq g_2$. Assume that $g_1$ and $g_2$ are
continuous on $\partial \Omega$. Let $u_i :  \Omega \cup \partial
\Omega \rightarrow \R$ be a continuous extension of $g_i$
satisfying the minimal equation $(\ref{minequ})$ on $\Omega$,
$i=1,2$, such that for any $p\in \pain \Omega$ we have
\begin{equation*}
\limsup _{q\to p}u_1(q)\leq g_1(p)\leq g_2(p)\leq \liminf _{q\to
p}u_2(q),
\end{equation*}
then we have $u_1\leqs u_2$ on $\Omega$.

We observe that this maximum principle holds  assuming the weaker
assumptions $\Mc (u_1) \geq 0$  and  $\Mc (u_2) \leq 0$ in
$\Omega$ (instead of $\Mc (u_1)=\Mc (u_2)=0$). }
\end{remark}

\smallskip

We shall need in the sequel the following important result of J.
Spruck.

\begin{remark}[{\bf {\em Spruck's result on graphs in $\hip^n\times \R$}}]\label{spr}
${}$ {\em We remark that among other pioneering  and general
results on $H$-graphs in $M\times \R$, J. Spruck obtained interior
a priori gradient
 estimates depending on a priori height estimates and the distance to the
boundary, \cite[Theorem 1.1]{Sp}. Combining this with classical
elliptic theory one obtains a {\em compactness principle}:  any
bounded sequence $(u_n)$ of solutions of Equation (\ref{minequ})
on a domain $\Omega \subset \hi n$ admits a subsequence that
converges uniformly on any compact subset of $\Omega$ to a
solution $u$ of Equation (\ref{minequ}) on $\Omega$.}
\end{remark}

\begin{lemma}[{\bf {\em Reflection principle for minimal graphs in
$\Hip^n\times\R$}}]\label{ref}

{} Let $\Omega\subset\hip^n$ be a domain whose boundary contains
an open set $V_\Pi$ of a geodesic hyperplane $\Pi$ of $\hip^n.$
Assume that $\Om$ is contained in one side of $\Pi$ and that
$\partial \Om\cap \Pi=\overl{V_\Pi}.$

Let $I$ be the reflection in $\hip^n$ with respect to $\Pi$ and
let $u:\Om\goto \R$ be a solution of the minimal equation
\eqref{minequ} that is continuous up to $V_\Pi$ and taking zero
boundary value data on $V_\Pi$. Then $u$ can be analytically
extended across $V_\Pi$ to a function $\wt{u}: \Om\cup V_\Pi\cup
I(\Om)\goto \R$ satisfying the minimal equation \eqref{minequ},
setting $\wt{u}=u(p),$ if $p\in \Om\cup V_\Pi$ and
$\wt{u}=-u(I(p)),$ if $p\in I(\Om).$
\end{lemma}

\begin{proof}
Without loss of generality, we will consider the upper half-space
model for $\hip^n.$
 Let $u:\Omega\subset\hip^n\goto \R$ be a $C^2$ solution
 of the minimal  equation
\eqref{minequ}.

We first note that the proof of the assertion does not depend on
the choice of the geodesic hyperplane $\Pi.$ Therefore, by
applying an ambient horizontal isometry to the minimal graph $G$,
if necessary, we may assume that, without loss of generality,
$\Pi=\{(x_1,x_2\ldots,x_n)\in\hip^n \mid x_1=0\}$ and we assume
that $\Om\subset \Pi^+:=\{(x_1,x_2\ldots,x_n)\in\hip^n \mid
x_1>0\}.$

Notice that  setting $w(x_1,x_2,\ldots,
x_n):=-u(-x_1,x_2,\ldots,x_n)$ for any $(x_1,\ldots,x_n)\in
I(\Om),$ then it is  simple to verify, on account of
\eqref{minequ},  that $w$ also satisfies  the minimal equation on
$I(\Om).$ Now let $p$ be an interior point of $V_\Pi$ and let
$B_r(p)\subset\hip^n$ be a small ball around $p$ of radius $r$
entirely contained in $\Om\cup V_\Pi\cup I(\Om).$ Let $\partial
B_r^+(p):=\partial B_r(p)\cap \Pi^+$ and let $f : \partial
B_r^+(p)\goto \R$ be the restriction of $u$ to $\partial
B_r^+(p).$  We now extend continuously $f$ to the whole sphere
$\partial B_r(p)$  of radius $r$ by odd extension. For simplicity
we still denote  this extension by $f$. We call $v$ the minimal
extension of $f$ on $B_r(p)$ given by Spruck \cite[Theorem
1.5]{Sp}, and also by the proof of Theorem
\ref{perron}-(\ref{Item.minequ}).
 Notice that the
maximum principle ensures that $v$ is the unique solution of the
minimal equation in $B_r(p)$ taking the continuous boundary value
data $f$ at $\partial B_r(p)$. Therefore we have
$v(-x_1,x_2,\ldots, x_n)=-v(x_1,\ldots,x_n)$ for any
$(x_1,\ldots,x_n)\in B_r(p)$ and thus $v(0,x_2,\ldots,x_n)=0$ for
any $(0,x_2,\ldots,x_n)\in V_\Pi$.

 The maximum principle again
guarantees that $v$ coincides with $u$ on $\Om \cap B_r(p)$, hence
the existence of the minimal extension of $f$ ensures the desired
analytic extension of $u$ to $B_r(p)$.
 This completes the proof.
\end{proof}

\section{Perron process for the minimal equation in $\hi n \times \R$}
\label{S.Perron}

%\subsection{\sc Perron process}

The notions of subsolution, supersolution and barrier for equation
(\ref{minequ}) are the same as in the two dimensional case, which
is treated with details by the authors in \cite{asi} and
\cite{SE-T}.

\begin{definition}[{\bf{\em Problem $(P)$}}]\label{probP}${}$
{\em In the product space $\hip^n\times \R,$ we consider the ball
model for the hyperbolic plane $\hip^n$. Let $\Om\subset \hip^n,$
be a domain.

%In $\overl{\hip^n}$ we set $\overline \Omega=\Om \cup \partial \Om
%\cup\partial_\infty\Om$
%  and  $\partial \overl{\Om}=\partial
%\Om\cup\partial_\infty\Om,$ where $\partial \Om\subset
%\hip^n$ and $\partial_\infty\Om\subset\partial_\infty
%\hip^n$.

 Let
$g:\partial \Om\cup\partial_\infty\Om \goto \R$ be a bounded
function. We consider the Dirichlet problem, say problem $(P)$,
for the vertical minimal hypersurface equation \eqref{minequ}
taking at any point of $\partial\Om\cup\partial_\infty \Om$
prescribed boundary (finite and asymptotic) value data $g.$  More
precisely,
\begin{equation*}
({P})\begin{cases} u\in C^2\left(\Om\right)  \text{ and } \ \mathcal M (u)= 0 \text{ in }\,\Om, \\
                          %   u = g \text{ on } \ \partial\Om\cup\partial_\infty \Om, \\
                                                         % &u\in C^2\left(\Om\right)
                             %C^0\left(\Om\,\mathop{\cup}\partial\Om\cup\partial_\infty \Om\setminus S
                            % \right)
                            \text{for any }  p\in \partial\Om\cup\partial_\infty
                            \Om  \text{ where } g \text{ is
                            continuous, } u \text{ extends} \\
                          \text{  continuously at } p \text{ setting }
                            u(p)=g(p).
            \end{cases}
\end{equation*}

Now, let $u:  \Om\cup\partial \Omega \rightarrow \R$ be a
continuous  function.

Let $U\subset \Omega$ be a closed round ball in $\hip^n$. We then
define the continuous function $M_{U}(u)$ on $\Omega \cup
\partial \Omega$ by:
\begin{align}\label{mu}
M_{U}(u)(x)\begin{cases}
      u(x) &\,\,\, \text{ if $x \in  \Omega \cup \partial \Omega\setminus U$}\\
   \tilde u (x)& \,\,\, \text{ if $x \in U$}
  \end{cases}
\end{align}
where $\tilde u$ is the minimal extension of $u_{\mid \partial U}$
on $\overline U$ given by Spruck \cite[Theorem 1.5]{Sp} and also
by the proof of Theorem \ref{perron}-(\ref{Item.minequ}).

\medskip

We say that $u$ is a {\em subsolution} (resp. {\em supersolution})
of $(P)$ if:
\begin{itemize}
\item [i)]\label{Item.round.disk} For any closed round ball
$U \subset \Omega$ we have\\ $u \leqs M_{U}(u)$ (resp. $u \geqs
M_{U}(u)$).

\vskip1mm

\item [ii)]  $u_{\mid \partial\Om}\; \leqs g$
(resp. $u_{\mid \partial\Om}\;\geqs g)$.

\vskip1mm

\item [iii)] We have
$\limsup_{q\to p} u(q) \leq g(p)$ (resp. $\liminf_{q\to p} u(q)
\geq g(p)$) for any $p\in \pain \Omega$.
\end{itemize}
}
\end{definition}

\begin{remark}\label{r1}{\em
 We now give some classical facts about subsolutions and
supersolutions (cf. \cite{C-H}, \cite{asi},\cite{SE-T}).
\begin{enumerate}
 \item  It is easily seen that if $u$ is $C^2$ on $\Omega$, the condition
 i) above  is equivalent to $\Mc (u) \geqs 0$ for subsolution or
$\Mc (u) \leq 0$ for supersolution.
 \item As usual if $u$ and $v$ are two subsolutions (resp.
supersolutions) of $(P)$ then $\sup (u,v)$ (resp. $\inf (u,v)$)
again is a subsolution (resp. supersolution).
 \item Also if $u$ is a subsolution (resp. supersolution) and
$U \subset \Omega$ is a closed round ball then $M_{U}(u)$ is again
a subsolution (resp. supersolution).

\item \label{Item.bounded} Let $\phi $  (resp. $u$) be a supersolution  (resp. a
subsolution) of problem $(P)$, then we have $u\leqs \phi$ on
$\Omega$. Moreover, for any closed round ball $U \subset \Omega$
we have $u\leqs M_U(u)\leqs M_U(\phi)\leqs \phi$.

\end{enumerate}
}
\end{remark}

\begin{definition}[{\bf{\em Barriers}}]\label{d1}{\em

We consider the Dirichlet problem $(P)$, see \defref{probP}. Let
$p \in \partial \Om\cup\partial_\infty\Om $ be a boundary point
where $g$ is continuous.

\begin{enumerate}
\item \label{D.Item.bar.1}
$\bullet$ Assume first that $p\in \partial \Omega$. Suppose that
for any $M>0$ and for any $k\in\Na$ there is an open neighborhood
$ \Nc_k$ of $p$ in $\hi n$ and a function $\omega_{k}^+$ $($resp.
$\omega_{k}^-)$  in $C^{2}(\Nc_k \cap \Omega)\cap
C^{0}(\overline{\Nc_k \cap \Omega)}$ such that
\begin{itemize}
\item[ i)]$\omega_{k}^+(x)_{\mid \partial \Omega \cap \overline {\Nc_k}}\geq g(x)$ and
$\omega_{k}^+(x)_{\mid \partial \Nc_k  \cap \Omega} \geq M$\\
$($resp. $\omega_{k}^-(x)_{\mid \partial \Omega \cap
\overline{\Nc_k}}\leq g(x)$ and $\omega_{k}^-(x)_{\mid
\partial \Nc_k  \cap \Omega} \leq -M)$.
\item[ii)] $\Mc (\omega_{k}^+) \leq 0$ $($resp.
$\Mc (\omega_{k}^-) \geq 0)$
 in $\Nc_k \cap \Omega$.

\item[ iii)] $\lim_{k \to +\infty}\omega_{k}^+(p)=g(p)$
 $($resp. $\lim_{k \to +\infty}\omega_{k}^-(p)=g(p))$.
\end{itemize}

\medskip

\noindent $\bullet$ If $p\in \partial_\infty \Om$, then we choose
for $\Nc_k$ an open set of $\hi n$ containing a half-space with
$p$ in its asymptotic boundary. We recall that a half-space is a
connected component of $\hi n\setminus \Pi$ for any geodesic
hyperplane $\Pi$. Then the functions $\omega_k^+$ and $\omega_k^-$
are in  $C^{2}(\Nc_k \cap \Omega)\cap C^{0}(\overline{\Nc_k \cap
\Omega})$and satisfy:

\begin{itemize}
\item[ i)]$\omega_{k}^+(x)_{\mid \partial \Omega \cap \overline {\Nc_k}}\geq g(x)$ and
$\omega_{k}^+(x)_{\mid \partial \Nc_k  \cap \Omega} \geq M$\\
$($resp. $\omega_{k}^-(x)_{\mid \partial \Omega \cap
\overline{\Nc_k}}\leq g(x)$ and $\omega_{k}^-(x)_{\mid
\partial \Nc_k  \cap \Omega} \leq -M)$.

\item[ ii)] For any $x \in \pain (\Omega \cap  \Nc_k)$ we
have $\liminf_{y\to x} \omega_k^+(y)\geq g(x)$ (for $y\in
\Nc_k\cap \Omega$)
 (resp. $\limsup_{y\to x} \omega_k^-(y)\geq g(x)$).

\item[iii)] $\Mc (\omega_{k}^+) \leq 0$ $($resp.
$\Mc (\omega_{k}^-) \geq 0)$
 in $\Nc_k \cap \Omega$.

\item[ iv)] $\lim_{k\to +\infty}\bigl(\liminf_{q\to
p}\omega_k^+(q)\bigr)=g(p)$ and \\
 $\lim_{k\to +\infty}\bigl(\limsup_{q\to
p}\omega_k^-(q)\bigr)=g(p)$.

\end{itemize}

\medskip

\item \label{D.Item.bar.2}
Suppose that  $p\in \partial \Omega$  and that  there exists a
supersolution $\phi$ $($resp. a subsolution $\eta)$ in
$C^2(\Om)\cap C^0(\overline \Om)$ such that $\phi(p)=g(p)$ (resp.
$\eta (p)=g(p)).$

\end{enumerate}
  In both cases (\ref{D.Item.bar.1}) or  (\ref{D.Item.bar.2})
we say that $p$ admits an {\em upper barrier} ($\omega_{k}^+, k\in
\Na$  or $\phi$) (resp. {\em lower} {\em barrier} $\omega_{k}^- ,
k\in \Na $ or $\eta$) for the problem ($P$). If $p$ admits an
upper and a lower barrier we say more shortly that $p$ admits a
{\em barrier}.}

\end{definition}

\begin{definition}[{\bf{\em $C^0$ convex domains}}]${}$

\begin{enumerate}
\item We say that  a $C^0$ domain $\Omega$ is  convex at $p\in\partial \Om$,
 if a neighborhood of $p$  in
$\ov \Omega$ lies in  one side of some geodesic hyperplane of $\hi
n$ passing through $p$.
\item We say that a $C^0$ domain  $\Omega$ is strictly convex at $p\in\partial \Om$
if a neighborhood $U_p\subset \ov \Omega$ of $p$  in $\ov \Omega$
lies in one side of some geodesic hyperplane $\Pi$ of $\hi n$
passing through $p$ and if $U_p \cap \Pi =\{p\}$.
\end{enumerate}

\end{definition}

We are then able to state the following result.

\begin{theorem}[{\bf {\em Perron process}}]\label{perron}

Let $\Omega \subset \hip^n$ be a domain and let \linebreak
 $g:
\partial\Om\cup\partial_\infty\Om \rightarrow \R$ be a  bounded
function. Let $\phi$ be a bounded supersolution of the Dirichlet
problem $(P)$, for example the constant function $\phi\equiv \sup
g$.

Set $\Sc_\phi = \{\vphi,\, \text{subsolution of }  \, (P),\, \,
\vphi \leq~ \phi \}$. We define for each $x \in \Omega $
$$
 u(x)=\sup_{\vphi \in \Sc_\phi}{\vphi(x)}.
$$
$($Observe that $\Sc_\phi \not=\emptyset$ since the constant
function $\varphi \equiv \inf g $ belongs to $\Sc_\phi$.$)$

 We have the following:
\begin{enumerate}
\item \label{Item.minequ} The function $u$ is $C^2$ on $\Omega$ and satisfies the
vertical minimal  equation $(\ref{minequ})$.

\item \label{Item.asym.bound} Let $p \in \partial_\infty\Om$ be an asymptotic
boundary point where $g$ is continuous. Then  $p$ admits a {\em
barrier} and therefore $u$ extends continuously at $p$ setting
$u(p)=g(p);$ that is, if $(q_m)$ is a sequence in $\hip^n$ such
that $q_m\goto p$, then $u(q_m)\goto g(p)$. In particular, if $g$
is continuous on $\partial_\infty\Om$ then the asymptotic boundary
of the graph of $u$ is  the restriction of the graph of $g$ to
$\partial_\infty\Om.$

\item \label{Item.fin.bound} Let  $p \in \partial\Om$ be a finite boundary
point where $g$ is continuous. Suppose
  that $p$ admits a
{\em barrier}. Then the solution $u$ extends continuously at $p$
setting $u(p)=g(p)$.

\item \label{bco} If $ \partial\Om$ is $C^0$ strictly convex at $p$
then $u$ extends continuously  at $p$ setting $u(p)=g(p).$
\end{enumerate}
\end{theorem}

\medskip

\begin{proof}
$ $ The proof of (\ref{Item.minequ}) follows as in \cite[Theorem
3.4]{asi}. We will give now some details. To obtain the solution
$u$ we need a {\em compactness principle} and we also need that
for any $y\in \Omega$ there exists a round closed ball $B\subset
\Omega$ such that $y\in \inter (B)$ and such that the Dirichlet
problem ($P$) can be solved on $B$ for any continuous boundary
data on $\partial B$.

 The compactness principle was shown by Spruck, see \cite{Sp}. The
 resolution of the Dirichlet problem on $B$ may also be
 encountered in \cite{Sp}, nevertheless we give some details for an
 alternative proof. Working in the half space model of $\hi n$, $B$
 can be seen as an Euclidean ball  centered at $y$ of radius
 $R>0$. Assume first that $h$ is a $C^{2,\alpha}$ function on
 $\partial B$. Observe that the eigenvalues of the symmetric
 matrix of the coefficients of $u_{x_i x_j}$ in Equation
(\ref{minequ}) are $1$ and $(W_M u)^2=1 +x_n^2(u_{x_1}^2+\cdots +
u_{x_n}^2)$, the  last with multiplicity $n-1$. Therefore, if $R$
is small enough, then the equation (\ref{minequ}) satisfies the
structure conditions  (14.33) in  \cite[Chapter 14]{GT}. Thus
Corollary 14.5 in \cite{GT} shows that there exist a priori
boundary gradient estimates. Then the classical elliptic theory
provides a $C^{2,\alpha}$ solution of ($P$), see for example
\cite[Chapter 11]{GT}. Finally, for continuous boundary data $h$
on $\partial B$,  we use an approximation argument.

\vskip2mm

 Let us proceed the proof of the assertion (\ref{Item.asym.bound}). Let
$p\in \pain \Omega$, we want to show that the minimal hypersurface
$M_1$ provides an upper and a lower  barrier at $p$. Let $k\in
\Na^\ast$, since $g$ is continuous at $p$, there exists a
neighborhood $U$ of $p$ in $ \hip^n \cup \pain \hi n$ such that
for any $q\in \bigl(\partial \Omega \cup \pain \Omega\bigr) \cap
U$ we have $g(p)-1/2k<g(q)<g(p)+1/2k$.

Let $\Pi$ be a geodesic hyperplane such that $\Pi \subset U$ and
such that the connected component of $\hip ^n\setminus \Pi$ lying
entirely in $U$ contains $p$ in its asymptotic boundary. We choose
an equidistant hypersurface $\Pi_k$ of $\Pi$ in the same connected
component of $\hip^n \setminus \Pi$. We denote by $\Nc_k$ the
connected component of $\hip^n\setminus \Pi_k$ containing $p$ in
its asymptotic boundary.

We can choose $\Pi_k$ such that there exist two copies $M_1^+$ and
$M_1^-$ of $M_1$ satisfying:
\begin{itemize}
\item $M_1^+$ takes the
asymptotic boundary value data $g(p)+1/2k$ on $\pain \Nc_k$, the
value data $+\infty$ on $\Pi$ and a finite value data
$A>\max\bigl( g(p)+1/2k, \sup_{\Omega} \phi\bigr)$ on $\Pi_k$.

\item $M_1^-$ takes the
asymptotic boundary value data $g(p)-1/2k$ on $\pain \Nc_k$, the
value data $-\infty$ on $\Pi$ and a finite value data $B<\inf g$
on $\Pi_k$.
\end{itemize}

Let us denote by $\om_k^+$ (resp. $\om_k^-$) the function on
$\Nc_k\cap \Omega$ whose graph is the copy $M_1^+$ (resp. $M_1^-$)
of $M_1$. We extend $\om_k^-$ on $\overline \Omega$ setting
$\om_k^-(q)=B$ for any $q\in \overline\Omega\setminus \Nc_k$,
keeping the same notation.

\medskip

\noindent {\bf Claim 1.} $\om_k^-\in \Sc_\phi$, that is $\om_k^-$
is a subsolution such that $\om_k^- \leq \phi$.

\noindent {\bf Claim 2.} For any subsolution $\varphi \in
\Sc_\phi$ we have $\varphi_{\mid \Nc_k\cap \Omega} \leq \om_k^+$.

\medskip

{ We assume momentarily that the two claims hold. We then have for
any $q\in \Nc_k \cap \Omega$: $\om_k^-(q)\leq u(q)$ (since
$\om_k^-\in \Sc_\phi$ and by the very definition of $u$) and $
\varphi(q)\leq \om_k^+ (q)$ for any subsolution $\varphi \in
\Sc_\phi$.} We deduce that
\begin{equation*}
\om_k^-(q)\leq u(q) \leq \om_k^+ (q)
\end{equation*}
for any $q\in  \Nc_k \cap \Omega$ and for any $k\in \Na^\ast$. The
rest of the argument is straightforward but we will provide the
details for the readers convenience.

 We thus have for any
$q\in \Nc_k \cap \Omega$:
\begin{equation*}
\om_k^-(q) -\bigl(g(p)-\frac{1}{2k}\bigr)-\frac{1}{2k}  \leq
u(q)-g(p) \leq \om_k^+ (q) -
\bigl(g(p)+\frac{1}{2k}\bigl)+\frac{1}{2k}.
\end{equation*}
Let $(q_m)$ be a sequence in $\Omega$ such that $q_m \to p$. By
construction,
% As
%$\om_k^+$ and $\om_k^-$ are continuous on $\bigl((\Nc_k \cap
%\Omega)\cup \pain (\Nc_k \cap \Omega)\bigr)\setminus \pain \Pi$,
for $m$ big enough we have $q_m \in \Nc_k\cap \Omega$ and
\begin{equation*}
\vert\om_k^+ (q_m) - \bigl(g(p)+\frac{1}{2k}\bigl)\vert \leq
\frac{1}{2k}, \quad \vert \om_k^- (q_m) -
\bigl(g(p)-\frac{1}{2k}\bigl)\vert \leq \frac{1}{2k}.
\end{equation*}
We then have $\vert u(q_m)-g(p)\vert \leq 1/k$ for $m$ big enough,
hence $u(q_m) \to g(p)$. We conclude therefore that $u$ extends
continuously at $p$ setting $u(p)=g(p)$.

\medskip

Let us prove Claim 1.
  By construction,
 $\om_k^-$ is continuous on $\overline \Omega$ and satisfies
  $\om_k^-\, _{\mid \partial \Omega}\leq g$ and
 $\limsup_{y\to p}\om_k^- (y) \leq g(p)$ ($y\in \overline \Omega$)
 for any $p\in \pain \Omega$. It is straightforward to show that for any closed
 round ball $U\subset \Omega$ we have $M_U (\om_k^-)\geq \om_k^-$,
 see (\ref{mu}) in Definition \ref{probP}.
  Hence $\om_k^-$ is a
subsolution of our Dirichlet problem ($P$). Observe that we have
 $\om_k^- \leq \phi$, see Remark \ref{r1}-(\ref{Item.bounded}),
 thus $\om_k^-\in \Sc_\phi$ as desired.

\smallskip

The proof of Claim 2 can be accomplished in the same way as the
proof of Claim 1, but we give another proof as follows. Let
$\varphi \in \Sc_\phi$. Assume by contradiction that $\sup_{\mid
\Nc_k\cap\Omega}(\varphi-\om_k^+)>0$. Since $\varphi$ and
$\om_k^+$ are bounded on $\Nc_k\cap \Omega$ we have $\sup_{\mid
\Nc_k \cap \Omega}(\varphi-\om_k^+)<+\infty$. Let $(q_m)$ be a
sequence in $\Nc_k\cap\Omega$ such that $(\varphi-\om_k^+)(q_m)\to
\sup_{\mid \Nc_k\cap\Omega}(\varphi-\om_k^+)$. Let $q\in \overline
{\Nc_k\cap\Omega} \cup \pain (\Nc_k\cap\Omega) $ be any limit
point of this sequence. Since
$$
\varphi\leq \phi <A=\om_k^+
$$
 on $\Pi_k$ and
$$
\varphi \leq g <g(p)+1/2k \leq \om_k^+
$$
on $\partial \Omega \cap \Nc_k$,
% and on
%$(\pain \Omega \cap  \overline\Nc_k) \setminus \pain \Pi_k $,
we must have
$$
q\in \Omega \cap \Nc_k \quad \text{or} \quad q\in \pain \Nc_k.
$$
The first possibility is discarded by the maximum principle. The
second possibility is also discarded since $\om_k^+\geq g(p)+1/2k$
on $\Nc_k$ and $\varphi (q_m)<g(p)+1/2k$ if $q_m \in \Nc_k \cap
\Omega $ is close enough of $\partial \Omega \cup \pain \Omega$.

\medskip

We conclude that $\om_k^+$ (resp. $\om_k^-$) is  an upper (resp. a
lower) barrier at any asymptotic point of $\Omega$ in the sense of
Definition \ref{d1}-(\ref{D.Item.bar.1}).

\medskip

We remark that the proof of the assertion (\ref{Item.fin.bound})
is analogous to the proof of the assertion
(\ref{Item.asym.bound}), see also \cite[Theorem 3.4]{asi}.

\medskip

Finally, the proof of the assertion (\ref{bco}) is a consequence
of the following.

\smallskip

\noindent {\bf Claim.} The family $M_d,\ d\in (0,1)$, provides a
barrier at any boundary point where $\Omega$ is strictly convex
and $g$ is  continuous.
\smallskip

 We proceed
 the proof of the claim as
follows. We choose the ball model for $\hip^ n$  and we may assume
that $p=0$. As $p$ is a strictly convex point, there is a geodesic
hyperplane $\Pi\subset\hip^ n$ such that,   locally, we have:

\smallskip

$\Pi\cap
\partial \Omega=\{0\}$ and, locally, $\Omega$ lies in one side, say $\Pi^ +$, of $\Pi.$

\smallskip

\noindent Let $M>0$ and $k\in \Na ^\ast$. We now construct a upper
barrier at $0$. Let $E(\rho)$ be the equidistant hypersurface to
$\Pi$ at distance $\rho$ lying in $\Pi^ +$. Let $E^+(\rho)$ be the
connected component of $\hip^ n\setminus E(\rho)$ that contains
$0.$  We call $\Nc$ the connected component of $E^+(\rho)\cap
\Omega$ such that  $0\in \overl{\Nc}.$ Consider the hypersurfaces
$M_d,\ d<1$, given by  equation (\ref{trl}). We choose $\rho>0$
such that $g(q) \leq g(0) +1/k$ on $\overl{\Nc}\cap \partial
\Omega$.

Using the $M_d$-Property (\ref{mdp}), we may choose $d$ near $1$,
$0<d<1$, such that $\lambda (\rho)>M-(g(0)-1/k)$. We set $w_k^+$
to be the function on $\overl{\Nc}$ whose the graph is (a piece
of) the vertical translated copy of $M_d$ by $g(0) +1/k$.

 Clearly,
the functions $w_k^+$ are continuous up to the boundary of $\Nc$
and give a upper barrier at $p$ in the sense of Definition
\ref{d1}-(\ref{D.Item.bar.1}). In the same way we can construct a
lower barrier at $p$. This completes the proof of the theorem.
\end{proof}

\section{Scherk type minimal hypersurfaces in $\hi n \times \R$}\label{S.Scherk}

\begin{definition}[\bf {\em Special rotational domain}]
\label{special domain} {\em  Let $\gamma, L \subset \hi n$ be two
complete geodesic lines with $L$ orthogonal to $\gamma$ at some
point $B\in \gamma \cap L$. { Using the half-space model for $\hi
n$, we can assume that $\gamma$ is the vertical geodesic such that
$ \pain \gamma = \{0,\infty\}$. We call $P\subset \hi n$ the
geodesic two-plane containing $L$ and $\gamma$. We choose $A_0\in
(0,B)\subset\gamma$ and $A_1 \in L\setminus \gamma$ and we denote
by $\alpha\subset P$ the euclidean segment joining $A_0$ and
$A_1$. Therefore the hypersurface $\Sigma$ generated by rotating
$\alpha$ with respect to $\gamma$ has the following properties.
\begin{enumerate}
\item $\inter (\Sigma)$ is smooth except   at point $A_0$.
\item $\Sigma$ is strictly convex in hyperbolic meaning and convex in euclidean
meaning.
\item \label{transv} $\inter (\Sigma)\setminus \{A_0\}$ is
transversal to the Killing field generated by the translations
along $\gamma$.
\end{enumerate}
Consequently $\Sigma$ lies in the mean convex side of the  domain
of $\hi n$ whose boundary is the hyperbolic cylinder with axis
$\gamma$ and passing through $A_1$. Let us call $\Pi\subset \hi n$
the geodesic hyperplane  orthogonal to $\gamma$ and passing
through $B$. Observe that the boundary of $\Sigma$ is a $n-2$
dimensional geodesic sphere of $\Pi$ centered at $B$.}

 We denote by
$U_\Sigma\subset \Pi$ the open geodesic ball centered at $B$ whose
boundary is the boundary of $\Sigma$. We call $\Dc_\Sigma\subset
\hi n$ the closed domain whose boundary is $U_\Sigma \cup \Sigma$.
Observe that $\partial \Dc_\Sigma$ is strictly convex at any point
of $\Sigma$ and convex at any point of $U_\Sigma$. Such a domain
will be called a {\em special rotational domain}. }
\end{definition}

\begin{proposition} \label{special t}  Let $\Dc_\Sigma\subset \hi n$ be a
special rotational domain.
 For any number $t\in\R$, there is a
unique solution $v_t$ of the vertical minimal equation in $\inter
(\Dc_\Sigma)$ which extends continuously to $\inter (\Sigma) \cup
U_\Sigma$, taking prescribed zero boundary value data on the
interior of $\Sigma$ and prescribed boundary value data $t$ on
$U_\Sigma$.

More precisely, for any $t\in \R$, the following Dirichlet problem
$(P_t)$ admits a unique solution $v_t$.
\begin{equation*}
({P}_t )\begin{cases} \mathcal M (u)= 0 \  \mathrm{ in }\ \inter (\Dc_\Sigma), \\
                             u = 0 \  \mathrm{ on }\ \inter (\Sigma), \\
                             u= t \  \mathrm{ on }\ U_\Sigma,\\
                             u\in C^2\left(\inter (\Dc_\Sigma)\right) \cap
                             C^0\left(\Dc_\Sigma \setminus  \partial
                             \Sigma\right).
            \end{cases}
\end{equation*}
Furthermore, the solutions $v_t$ are strictly increasing with
respect to $t$ { and satisfy $0 < v_t <t$ on $ \inter
(\Dc_\Sigma)$}.
\end{proposition}

\begin{proof}

 Before beginning the proof of the existence part of the
statement, we would like to remark that, as the ambient space has
dimension $n$ (arbitrary), we cannot use classical Plateau type
arguments to obtain a regular minimal hypersurface in $\mathbb
H^n\times \mathbb R$ whose the boundary is $\bigl( \textbf{}\Sigma
\times \{0\}\bigr) \cup \bigl(U_\Sigma \times \{t\}\bigr)\cup
\bigl(\partial \Sigma \times [0,t]\bigr)$.

We are not able to apply directly {\em Perron process}
(\thmref{perron}) to solve this Dirichlet problem.  For this
reason, in order to prove the existence part of our statement, we
need to consider an auxiliary Dirichlet problem, as follows.

We can assume that $t>0$. For $k \in \Na^\ast$ we set
\begin{equation*}
V_k :=\{p\in \Sigma \mid \text{dist\,}(p, \Pi)\leqs \frac{1}{k}\},
\end{equation*}
where we recall that $\Pi\subset \hi n$ is the geodesic hyperplane
containing $U_\Sigma$ and where $dist$ means the distance in $\hi
n$.

We choose a translated copy $M_{d_k}$ of the hypersurface $M_d$,
see section \ref{S.Md}, with $d_k<1$, given by a function
$\lambda_k (\rho)$ satisfying $\lambda_k (0)=t$ and $\lambda_k
(1/k)\leq -1$. Since $\lambda_k$ is an odd function for $d_k\in
(0,1)$, the $M_d$-Property (\ref{mdp}) insures that such a
$M_{d_k}$ exists for $d_k<1$ close enough to $1$. Then  we choose
a continuous function $f_k:V_k \rightarrow [0,t]$ such that

\begin{enumerate}

\item $f_k=t$ on $ \partial \Sigma=V_k \cap \Pi$.

\item $f_k=0$ on $ \partial V_k \cap \inter (\Sigma)$.

\item The graph of $f_k$  stands above the hypersurface $M_{d_k}$, that is
$f_k \geq \lambda_k$ on $V_k$.
\end{enumerate}

 Now we define a function $g_k : \partial \Dc_\Sigma
\rightarrow [0,t]$ setting:
\begin{equation*}
g_k(p)=\begin{cases}        0& \text{ if}\, p\in \Sigma \setminus V_k, \\
                         f_k & \text{if}\, p\in V_k, \\
                            t& \text{ if }\, p\in U_\Sigma. \\
      \end{cases}
\end{equation*}
Note that $g_k$ is a continuous function on $\partial \Dc_\Sigma$.
Then we consider an auxiliary Dirichlet problem $(\widehat{P}_k)$
as follows:
\begin{equation*}
(\widehat{P}_k)\begin{cases} \mathcal M (u)= 0 \text{ in }\,\inter (\Dc_\Sigma), \\
                             u = g_k  \text{ on }\,  \partial \Dc_\Sigma, \\
                            u\in C^2\left(\inter (\Dc_\Sigma)\right)
                             \cap C^0\left(\Dc_\Sigma\right).
            \end{cases}
\end{equation*}

\medskip

Observe that the hypersurface $M_{d_k}$ provides a lower barrier
at any point of $U_\Sigma$ and that at such a point the constant
function $\omega^+ \equiv t$ is an upper barrier in the sense of
Definition \ref{d1}-(\ref{D.Item.bar.2}). Furthermore, $\partial
\Dc_\Sigma$ is $C^0$ strictly convex at any other point, that is
at any point of  $\Sigma$. Therefore the hypersurfaces $M_d$,
$d<1$, provide a barrier at these points, see the proof of
\thmref{perron}-(\ref{bco}). Thus, any point of $\partial
\Dc_\Sigma$ has a barrier. Applying {\em Perron Process}
(\thmref{perron}), considering the set of subsolutions to problem
$(\widehat{P}_k)$ below the constant supersolution identically
equal to $t$, we find a solution $w_k$ of the Dirichlet problem
$(\widehat{P}_k)$. Observe that the zero function is a subsolution
of $(\widehat{P}_k)$. Therefore we have $0\leqs w_k \leqs t$ for
any $k>0$.

Using the {\em reflection principle} with respect to $\Pi$
(\lemref{ref}), it follows that each point of $U_\Sigma$ can be
considered as an interior point of the domain of a function,
denoted again by $w_k$, satisfying the minimal equation, bounded
below by $0$ and bounded above by $2t$. Observe that this estimate
is independent of $k>0$.

Consequently, using the {\em compactness principle}, we can find a
subsequence that converges to a function $v_t \in C^2(\inter
(\Dc_\Sigma))\cap C^0(\inter (\Dc_\Sigma)\cup U_\Sigma)$
satisfying the minimal equation $\mathcal M (v_t)=0$ and such that
$v_t(p)=t$ at any $p\in U_\Sigma$. Since any point of $\inter
(\Sigma)$ has a barrier the function $v_t$ extends continuously
there, setting $v_t(p)=0$ at any $p\in \inter (\Sigma)$. We have
therefore proved the existence of a solution $v_t$ of the
Dirichlet problem $(P_t)$. { Observe that by construction we have
$0<v_t <t$ on  $\inter (\Dc_\Sigma)$.}

{  Let us prove now uniqueness of the solution of $(P_t)$. Let $u$
and $v$ be two solutions of the Dirichlet problem $(P_t)$. We will
adapt the proof of \cite[Theorem 2.2]{HRS} to our situation.}

We are going to use the notations of Definition \ref{special
domain}. Let us recall that $P$ is the geodesic two-plane
containing the geodesic lines $\gamma$ and $L$. Let $\varepsilon
>0$ and let us call $c_\varepsilon \subset P$ the intersection of
the circle or radius $\varepsilon$ centered at $A_1$ with the
compact subset of $P$ delimited by $\gamma, L$ and the euclidean
segment $\alpha$. We denote by $C_\varepsilon \subset \hi n $ the
compact hypersurface obtained by rotating $c_\varepsilon$ with
respect to $\gamma$. Let $V_\varepsilon$ be the $n-1$ volume of
$C_\varepsilon $. Observe that $V_\varepsilon \to 0$ when
$\varepsilon \to 0$. From now the arguments follow as in
\cite{HRS}, so we just sketch the proof.

  For $N>0$ large we define
\begin{equation*}
 \varphi =\begin{cases}
N-\varepsilon   & \text{if}\quad u-v \geq N \\
u-v-\varepsilon & \text{if}\quad \varepsilon<u-v <N \\
0 & \text{if}\quad u-v\leq\varepsilon
          \end{cases}
\end{equation*}
Let us call $\Dc_\varepsilon$ the connected component of
$\Dc_\Sigma \setminus C_\varepsilon$ containing $A_0$ (we have
$\Dc_\varepsilon \to \Dc_\Sigma$ when $\varepsilon \to 0$).
Observe that $\varphi \equiv 0$ along $\partial \Dc_\varepsilon
\setminus C_\varepsilon$. So that, applying the divergence theorem
and using the fact that $u$ and $v$ are solutions of the minimal
graph equation, we obtain
\begin{equation*}
 \int_{C_\varepsilon} \varphi \langle
\frac{\nabla u}{W_M u}-\frac{\nabla v}{W_M v}, \nu\rangle ds =
\int_{\Dc_\varepsilon} \langle \nabla \varphi, \frac{\nabla u}{W_M
u}-\frac{\nabla v}{W_M v}\rangle dV
\end{equation*}
where $\nu$ is the exterior normal to $\partial C_\varepsilon$. It
is shown in \cite[Lemma 2.1]{HRS} that $\langle \nabla u -\nabla
v, \frac{\nabla u}{W_M u}-\frac{\nabla v}{W_M v}\rangle \geq 0$
with equality at a point if, and only if, $\nabla u=\nabla v$.
Therefore
\begin{align*}
 0 \leq \int_{\Dc_\varepsilon}
\langle \nabla \varphi, \frac{\nabla u}{W_M u}-\frac{\nabla v}{W_M
v}\rangle dV
 = & \int_{C_\varepsilon} \varphi \langle
\frac{\nabla u}{W_M u}-\frac{\nabla v}{W_M v}, \nu\rangle ds \\
\leq & \ 2 N V_\varepsilon
\end{align*}
Letting $\varepsilon \to 0$, we get that $\nabla u \equiv \nabla
v$ in the set where $0 < u-v < N$. Letting $N \to +\infty$ we
obtain that $\nabla u \equiv \nabla v$ in the set $\{ u>v\}$.
Assume that $\inter \{ u>v\} \not=\emptyset$, then there exists a
constant $\lambda >0$ such that $u=v + \lambda$ on an open subset
of $\Dc_\Sigma$. By analyticity we deduce that $u=v + \lambda$
everywhere on $\Dc_\Sigma \setminus \partial \Sigma$, which is
absurd since $u=v$ on $\partial \Dc_\Sigma \setminus \partial
\Sigma$. Therefore we get that $\inter \{ u>v\} =\emptyset$, that
is $u\leq v$ on $\Dc_\Sigma \setminus \partial \Sigma$. The same
argument shows also that $v\leq u$ on $\Dc_\Sigma \setminus
\partial \Sigma$. Therefore $u=v$ and the proof of the uniqueness
of the solution of Dirichlet problem $(P_t)$ is completed.

\medskip

At last, let us prove that the family $\{v_t\}$ of the solutions
of Dirichlet problem $(\Pc_t)$ is strictly increasing on $t$. We
could adapt the same arguments of \cite[Theorem 2.2]{HRS} as
before, but we will give another proof.

Let $0< t_1 <t_2$ and let $v_1$ and $v_2$ be the solutions of the
Dirichlet problems $(P_{t_1})$ and $(P_{t_2})$ respectively. Let
$p$ be a fixed arbitrary point in the interior of $\Dc_\Sigma.$

For $\e$ small enough consider a $\e$-translated copy  of the
graph of $v_1$ along $\ga$ in the orientation $A_0\rightarrow B$.
This graph  is given by a function $v_1^\e$ over a translated copy
$\Dc_\Sigma(\e)$ of $\Dc_\Sigma$. Taking into account the
properties  on $\Sigma$ stated  in Definition \ref{special
domain}, we have $\Dc_\Sigma(\e) \cap \Sigma=\emptyset$. We may
assume that $\e$ is chosen  small so  that $p$ belongs to $\inter
(\Dc_\Sigma(\e))$. Since $0<v_1 <t_1$ on $\inter \Dc_\Sigma$, we
get that $v_1^\e$ is less than $v_2$ along the boundary of
$\Dc_\Sigma\cap \Dc_\Sigma(\e)$. Using maximum principle  we
deduce that $v_1^\e(p)<v_2(p)$, for $\e$ small enough, since
$v_1^\e<v_2$ along $\partial \bigl(\Dc_\Sigma\cap
\Dc_\Sigma(\e)\bigr)$. Thus letting $\e\to 0$ we have therefore
that $v_1(p)\leqs v_2(p)$, this accomplishes the proof.

\end{proof}

%\bigskip

\begin{theorem}[\bf {\em Rotational Scherk hypersurface}] \label{special Scherk}
Let $\Dc_\Sigma\subset \hi n$ be a special rotational domain.
There is a unique solution $v$ of the vertical minimal equation in
$\inter (\Dc_\Sigma)$  which extends continuously to $\inter
(\Sigma)$, taking prescribed zero boundary value data and taking
boundary value $\infty$ for any approach to  $U_\Sigma$.

More precisely,  the following Dirichlet problem $(P)$ admits a
unique solution $v_\infty$.
\begin{equation*}
({P} )\begin{cases} \mathcal M (u)= 0 \  \mathrm{ in }\  \inter (\Dc_\Sigma), \\
                             u = 0 \  \mathrm{ on }\  \inter (\Sigma), \\
                             u= +\infty \ \mathrm{ on }\  U_\Sigma,\\
                             u\in C^2\left(\inter (\Dc_\Sigma)\right) \cap
                              C^0\left(\Dc_\Sigma \setminus
\ov U_ \Sigma \right).
            \end{cases}
\end{equation*}
We call the graph of $v$ in $\hi n \times \R$ a {\em rotational
Scherk hypersurface}.

\end{theorem}

\begin{proof}

First, we will prove the existence part of the Theorem. We
consider the family of functions $v_t,\, t>0$, given by
\propref{vt}. Recall that  $\Pi\subset \hi n$ is the totally
geodesic hyperplane containing $U_\Sigma$. We consider a suitable
copy of $M_1$ (see section \ref{S.Md}) as barrier as follows:
choose $M_1$ such that $M_1$ is a graph of a function $u_1$ whose
domain is the component of $\Hip^n\setminus \Pi$ that contains
$\Dc_\Sigma$, with $u_1$ taking boundary value data $+\infty$ on
$\Pi$ and  taking zero asymptotic boundary value data. By applying
maximum principle we have that $u_1(p)> v_t(p)$ for all $p\in
\Dc_\Sigma$ and all $t>0$.

 Using {\em compactness principle}  we
obtain that a subsequence of the family converges uniformly on any
compact subsets of $\inter (\Dc_\Sigma)$ to a solution $v_\infty$
of the minimal equation. Since the family is strictly increasing
$v_\infty$ takes the value $+\infty$ on $U_\Sigma$. That is, for
any sequence $(q_k)$ in $\inter (\Dc_\Sigma)$ converging to some
point of $U_\Sigma$ we have $v_\infty (q_k)\to +\infty$.

 Let $p\in \inter (\Sigma)$, since $\partial \Dc_\Sigma$ is
 $C^0$ strictly convex at $p$, the hypersurfaces $M_d$, $d<1$, provide a barrier at
 $p$, see the proof of \thmref{perron}-(\ref{bco}).
 Consequently $v_\infty$ extends continuously at $p$ setting
 $v_\infty  (p)=0$. Therefore $v_\infty$ is a solution of the
 Dirichlet problem ($P$).

The proof of uniqueness of $v_\infty$ proceeds in the same way as
the proof of  the monotonicity of the family $\{v_t\}$ in
Proposition \ref{special t}. This completes the proof of the
Theorem.
\end{proof}

\bigskip

\begin{theorem}[{\bf{\em Barrier at a $C^0$ convex
point}}]\label{Sch.B} Let $\Omega\subset \hi n$ be a domain and
let $p_0\in \partial \Omega$ be a boundary point where $\Omega$ is
$C^0$  convex. Then for any bounded data $g:\partial \Omega\cup
\pain \Omega \rightarrow \R$ continuous at $p_0$, the family of
rotational Scherk hypersurfaces  provides a barrier at $p_0$ for
the Dirichlet problem $(P)$. In particular, in {\em Theorem}
$\ref{perron}$-$(\ref{bco})$ the assumption {\em $C^0$ strictly
convex} can be replaced  by {\em $C^0$ convex}.
\end{theorem}

\begin{proof}

We use the same notations as in the definition of a special
rotational domain, Definition \ref{special domain}.

 We will prove that the rotational Scherk hypersurfaces with
$-\infty$ boundary data on the boundary part $U_\Sigma$  provide
an upper barrier at  $p_0$. For the lower barrier the construction
is similar.

Let $\Dc_\Sigma$ be a special rotational domain. Let $\om$ be the
height function of the rotational Scherk hypersurface $S$ taking
$-\infty$ boundary data on $U_\Sigma$ and $0$ boundary data on the
interior of $\Sigma$, given by Theorem \ref{special Scherk}.

\medskip

\noindent {\bf Claim 1.} $\om$ is decreasing along the oriented
geodesic segment  $[A_0, B]\subset \gamma$ (going from $A_0$ to
$B$).

\smallskip

\noindent {\bf Claim 2.} Let $D$ be any point on the open geodesic
segment $(A_0,B)$, and let $\beta\subset \Dc_\Sigma$ be a geodesic
segment issuing from $D$, ending at some point $C\in \inter
(\Sigma)$ and orthogonal to $[A_0, B]$ at $D$.

Then $\om$ is increasing along $\beta =[D,C]$, oriented from $D$
to $C$.

\bigskip

We first prove the theorem assuming that the two claims  hold.

\medskip

Let $D\in (A_0,B)$ and let $\Pi_D\subset \hip^n$ be the geodesic
hyperplane through $D$ orthogonal to the geodesic segment $[A_0,
B]$. Let $\Dc_\Sigma^+$ be the connected component of $\Dc_\Sigma
\setminus \Pi_D$ containing the point $A_0$. Let $q$ be any point
belonging to the closure of $\Dc_\Sigma^+$. The claims ensure that
$\om (q)\geq \om (D)$.

\medskip

Let $p_0\in\partial \Om$ be a $C^0$  convex point and let $g$ be a
bounded data continuous at $p_0$. Let $M>0$ be any positive real
number. It suffices  to show that for any $k\in\Na^\ast$ there is
an open neighborhood $ \Nc_k$ of $p_0$ in $\hip^n$   and a
function $\omega_{k}^+$   in $C^{2}(\Nc_k \cap \Omega)\cap
C^{0}(\overline{\Nc_k \cap \Omega)}$ such that
\begin{itemize}
\item[ i)]$\omega_{k}^+(x)_{\mid \partial\Om \cap \Nc_k}\geq g(x)$ and
$\omega_{k}^+(x)_{\mid \partial \Nc_k  \cap \Omega} \geq M$,
\item[ii)] $\mathcal{M} (\omega_{k}^+) = 0$
 in $\Nc_k \cap \Omega$,
\item[ iii)] $\omega_{k}^+(p_0)=g(p_0) +1/k$.
\end{itemize}

 By continuity there exists $\ep>0$ such that for any $p\in\partial
\Om$ with $\dist(p, p_0) <\ep$ we have $g(p) <g(p_0) + 1/k.$

By assumption there exist a geodesic hyperplane $\Pi_{p_0}$
through $p_0$ and an open neighborhood $W\subset \Pi_{p_0}$ of
$p_0$  such that $W\cap\Om=\emptyset$.
 We set
$\Omega_\varepsilon=\{p\in \Omega \mid \dist
(p_0,p)<\varepsilon\}$. Up to choosing $\varepsilon$ small enough,
we can assume that $\Omega_\varepsilon$ is entirely contained in a
component of $\hi n \setminus \Pi_{p_0}$. Let $\gamma$ be the
geodesic through $p_0$ orthogonal to $\Pi_{p_0}$.

We choose a special rotational domain $\Dc_\Sigma$ such that:
\begin{itemize}
\item the hyperplane $\Pi$ is
orthogonal to $\gamma$, (recall that $U_\Sigma \subset \Pi$)
\item the diameter of $\Dc_\Sigma$ is lesser than $\frac{\ep}{4}$,
\item $\ov{\Om}\cap U_\Sigma=\emptyset$,
\item  $A_0\in \gamma$, $\dist(p_0,A_0) < \frac{\ep}{8}$ and $A_0$ belongs to
the same component of $\hi n \setminus \Pi_{p_0}$ than
$\Omega_\varepsilon$.
\end{itemize}

 Let $M^\prime>\max\{M, g(p_0) +1/k\}.$ We consider the rotational
Scherk hypersurface (graph of $\om$) taking $M^\prime$ boundary
value data on the interior of $\Sigma$ and $-\infty$ on
$U_\Sigma$. By continuity, there exists a point $p_1\in\ga$ where
$\om (p_1)=g(p_0) +1/k.$ Up to a horizontal translation along
$\ga$ sending $p_1$ to $p_0$, we may assume that $\om(p_0)=g(p_0)
+1/k.$ Then we set $\Nc_k=\inter (\Dc_\Sigma)\cap\Om$ and
$\om_k^+=\om_{\mid \Nc_k}$, the restriction of $\om$ to $\Nc_k$.
Therefore we have $\omega_{k}^+(x)_{\mid \partial \Nc_k  \cap
\Omega}=M^\prime \geq M$, furthermore Claim 1 and Claim 2 show
that $\omega_{k}^+(x)_{\mid
\partial\Om \cap  \Nc_k}\geq g(p_0) +1/k \geq g(x)$, as desired.

\bigskip

We now proceed to the proof of Claim 1.  Let $p_1,p_2\in (A_0,B)$
with $p_1<p_2$, we want to show that $\om (p_1)\geq \om (p_2)$.
Let $p_3 \in (p_1,p_2)$ be the middle point of $p_1$ and $p_2$ and
let $\Pi_{p_3}\subset \hi n$ be the geodesic hyperplane through
$p_3$ orthogonal to $(A_0,B)$. We denote by $\sigma$ the
reflection in $\hip^n$ with respect to $\Pi_{p_3}$. Let
$\Dc_\Sigma^+$  be the connected component of $\Dc_\Sigma
\setminus \Pi_{p_3}$ containing $A_0$ and let $\Dc_\Sigma^-$ be
the other component. We denote by $S^+$ the part of the rotational
Scherk hypersurface which is a graph over $\Dc_\Sigma^+$. Observe
that the definition of a special rotational domain ensures that
$\sigma (\Dc_\Sigma^+)\cap \Sigma=\emptyset$. Hence a part of
$\sigma (S^+)$ is the graph of a function $v$ over a part $W$ of
$\Dc_\Sigma^-$ such that $v\geq \om$ on $\partial W$. We conclude
therefore with the aid of the maximum principle that $v\geq \om$
on $W$. This shows that $\om (p_1)\geq \om (p_2)$ as desired.

\medskip

Now let us prove Claim 2. Let $q_1,q_2 \in [D,C]$ with $q_1<q_2$,
we want to show that $\om (q_1) \leq \om (q_2)$. Let $q_3 \in
(q_1,q_2)$ be the middle point of $q_1$ and $q_2$ and let
$\Pi_{q_3}$ be the geodesic hyperplane through $q_3$ orthogonal to
$[D,C]$. Let $\sigma$ be the reflection in $\hip^n$ with respect
to $ \Pi_{q_3}$. Let $\Dc_\Sigma^-$ be the connected component of
$\Dc_\Sigma \setminus  \Pi_{q_3}$ containing $A_0$ and let
$\Dc_\Sigma^+$ be the other component.

\smallskip

\noindent {\bf Assertion.} If $U_\Sigma \cap
\Pi_{q_3}\not=\emptyset$ then there exists a point $X_0\in
U_\Sigma\cap \Dc_\Sigma^+$ such that $\sigma (X_0) \not\in
\Dc_\Sigma $.

\smallskip

We assume this assertion for a while. If $U_\Sigma \cap
\Pi_{q_3}\not=\emptyset$ then for any $Z\in U_\Sigma\cap
\Dc_\Sigma^+$, with $Z\not\in  \Pi_{q_3}$, we have $\sigma
(Z)\not\in \Dc_\Sigma$. Indeed, if not, since $\sigma (X_0)
\not\in \Dc_\Sigma $, we would find by continuity a point $Y\in
U_\Sigma\cap \Dc_\Sigma^+$, with $Y\not\in \Pi_{q_3}$, such that
$\sigma (Y)\in \Pi$ and $\sigma (Y)\not=Y$. Therefore the geodesic
segment $[Y, \sigma(Y)]$ is globally invariant with respect to
$\sigma$. Thus $[Y, \sigma(Y)]$ is orthogonal to $\Pi_{q_3}$ and
therefore $\Pi$ is also orthogonal to $\Pi_{q_3}$. Hence, we
conclude that the whole hyperplane $\Pi$ is invariant by the
reflection $\sigma$, which contradicts the assertion.

We denote by $\Sigma^-$ the connected component of
$\Sigma\setminus \Pi_{q_3}$ which contains $A_0$ and we denote by
$\Sigma^+$ the other component.

Observe that for any $p\in \Sigma^+$ we have $\sigma (p) \not\in
\Sigma^-$. Indeed,  assume first that $p$ lies in the euclidean
segment $\alpha \subset P$ (see Definition  \ref{special domain}).
By construction, $\sigma (p)$ belongs to the equidistant curve
$E_p \subset P$, passing through $p$, of the geodesic line
$\Gamma$ containing the segment $[D,C]$. Recall that $\Gamma$ and
$E_p$ have the same asymptotic boundary. Furthermore, $E_p$ is
symmetric with respect to any geodesic hyperplane orthogonal to
$\Gamma$. Since $\Dc_\Sigma$ is symmetric with respect to the
geodesic hyperplane through $D$ orthogonal to $\Gamma$, we have
that $\sigma (p)\not\in \Sigma^-$. Assume now that $p\in \Sigma^+
\setminus \alpha$. Let us denote by $V$ the 3-dimensional geodesic
submanifold of $\hi n$ containing $p$ and the geodesic two-plane
$P$. Let $H_D\subset \hi n$ be the geodesic hyperplane through $D$
orthogonal to the geodesic $\Gamma$. Then the symmetric of $p$
with respect to  $H_D$, denoted by $p^\ast$, is the same than the
symmetric of $p$ in $V$ with respect to the geodesic two-plane
$V\cap H_D$. As before, $\sigma (p)$ belongs to the equidistant
curve $E_p \subset P$, passing through $p$, of the geodesic line
$\Gamma$. Furthermore $E_p$ is symmetric with respect to the
geodesic hyperplanes $H_D$ and $\Pi_{q_3}$. Now $E_p$ is an arc of
circle passing through $p$ with the same asymptotic boundary than
$\Gamma$. As $\Dc_\Sigma \cap V$ is a compact part of an euclidean
cone we get that $E_p \cap \Sigma =\{p, p^\ast\}$. Since
$\sigma(p) \not=p^\ast$, we conclude that $\sigma (p)\not\in
\Sigma^-$.

\smallskip

Thus the reflected of $\partial \Dc_\Sigma^+ $ by $\sigma$ does
not have any intersection with $\Sigma^-$. We denote by $S^+$ the
part of the rotational Scherk hypersurface which is a graph over
$\Dc_\Sigma^+$.
 Hence a part of
$\sigma (S^+)$ is the graph of a function $v$ over the domain
$W=\sigma (\Dc_\Sigma^+) \cap \Dc_\Sigma^-$ such that $v\geq \om$
on $\partial W$.  We now are able to conclude the proof of Claim
2, assuming the assertion, by applying the maximum principle, to
infer that $\om (q_2)\geq \om (q_1)$.

\smallskip

Finally, if $U_\Sigma \cap \Pi_{q_3}=\emptyset$ by a similar and
simpler argument we complete the proof of Claim 2.

\medskip

To prove the assertion, let us denote by $P_C\subset \hip^n$ the
geodesic two-plane containing the geodesic segments $[A_0,B]$ and
$[D,C]$. Thus $P_C$ is orthogonal to $\Pi_{q_3}$, since it
contains $[D,C]$, and is orthogonal to $\Pi$, since it contains
$[A_0,B]$. We consider the open geodesic segment  $\gamma_1=P_C
\cap U_\Sigma$ and the geodesic line $\gamma_2=P_C \cap
\Pi_{q_3}$. Assume that $U_\Sigma \cap \Pi_{q_3}\not=\emptyset$.
Then, since $P_C$ is orthogonal to $\Pi$ and to $\Pi_{q_3}$  we
have $\gamma_2 \cap U_\Sigma \not=\emptyset$. Therefore $\gamma_2$
intersects $\gamma_1$ at some point $\{z\}=\gamma_1 \cap
\gamma_2$.

 Observe that the points $D, q_3,z$ and $B$ define a
geodesic quadrilateral $\Qc$ in $P_C$ with right angles at
vertices $B,D$ and $q_3$. Therefore the interior angle of $\Qc$ at
$z$ is strictly smaller than $\pi/2$. Let us denote by
$\gamma_1^+\subset \gamma_1$ the connected component of
$\gamma_1\setminus \{z\}$ which does  not contain $B$. Observe
that $\gamma_1^+\subset U_\Sigma\cap \Dc_\Sigma^+$.
 Let $s$ be the reflection in $P_C$ with respect to
$\gamma_2$. Then $s(\gamma_1^+)$ does not have intersection with
$\Dc_\Sigma$, $s(\gamma_1^+)\cap \Dc_\Sigma=\emptyset $. Since
$P_C$ is orthogonal to $\Pi_{q_3}$ we have that
$s(\gamma_1^+)=\sigma(\gamma_1^+)$. Therefore for any $X\in
\gamma_1^+$ we have $\sigma (X) \not\in \Dc_\Sigma$ as claimed,
this completes the  proof.
\end{proof}

\begin{definition}[\bf {\em Independent points and admissible polyhedra}]\label{pol}
${}$
\begin{enumerate}
\item We say that $n+1$ points $A_0,\ldots,A_n$ in $\hi n$ are
{\em independent} if there is no geodesic hyperplane containing
these points. If  $A_0,\ldots,A_n$ in $\hi n$ are independent then
we remark that any choice of $n$ points among them determines a
unique geodesic hyperplane of $\hi n$.

\item Let $A_0,\ldots,A_n$ be $n+1$ independent points in $\hi n$.
We call $\Pi_i$ the geodesic hyperplane containing these points
excepted $A_i$, $i=0,\ldots, n$ and we call $\Pi_i^+$ the closed
half-space bounded by $\Pi_i$ and containing $A_i$. Then the
intersection of these half-spaces is a polyhedron $\Pc$: the
convex closure of $A_0,\ldots,A_n$. The boundary of $\Pc$ consists
of $n+1$ closed faces $F_i\subset \Pi_i$, the face $F_i$ contains
in its boundary all the points $A_0,\ldots,A_n$ excepted $A_i$. We
call such a polyhedron an {\em admissible polyhedron}.
\end{enumerate}
\end{definition}

\begin{corollary} \label{vt}  Let $\mathcal P$ be an admissible
polyhedron. For any number $t\in\R,$ there is a unique solution
$v_t$ of the vertical minimal equation in $\mathrm{int} (\Pc)$
which extends continuously to $\partial \Pc\setminus
\partial F_0,$ taking prescribed zero boundary value data on
$F_1\setminus
\partial F_0,\ldots ,F_n\setminus \partial F_0$ and prescribed
boundary value $t$ on $\mathrm{int}(F_0).$ More precisely, for any
$t\in \R$, the following Dirichlet problem $(P_t)$ admits a unique
solution $v_t$.
\begin{equation*}
({P}_t )\begin{cases} \mathcal M (u)= 0\  \mathrm{ in }\ \mathrm{int}(\Pc), \\
                             u = 0 \  \mathrm{ on }\  F_j\setminus \partial F_0,\,
j=1,\ldots,n, \\
                             u= t \  \mathrm{ on }\ \mathrm{int}
                             (F_0), \\
                             u\in C^2\left(\mathrm{int} (\Pc)\right) \cap
                             C^0\left(\Pc \setminus  \partial
                             F_0\right).
            \end{cases}
\end{equation*}
Furthermore, the solutions $v_t$ are strictly increasing with
respect to $t$  and satisfy $0< v_t <t$ on $\inter (\Pc)$.
\end{corollary}

\begin{proof}
The  existence part of the statement  is  a consequence of Theorem
\ref{Sch.B}.

 The uniqueness is proved in the same way as in
Proposition \ref{special t}.

To prove the monotonicity of the family $\{v_t\}$ we consider a
point $q\in \inter (F_0)$. Notice that $\partial \Pc$ is
transversal to the Killing field generated by translations along
the geodesic line $\ga$ containing $A_0$ and $q$. Then the proof
proceeds as in the proof of Proposition \ref{special t}.
\end{proof}

Using the above proposition we are able to construct a {\em
 Scherk type  minimal hypersurface} in $\hip^{n}\times \R.$

\begin{theorem}[{\bf {\em First Scherk type hypersurface in $\hip^n\times
 \R$}}]\label{Sch1}
  Let $\mathcal P$ be an admissible convex polyhedron. There
  is a unique solution $v_\infty$ of the minimal
equation in $\inter (\Pc)$ extending continuously up to
 $\partial \Pc\setminus  F_0$, taking prescribed zero
boundary value data on $F_1\setminus \partial F_0,\ldots
,F_n\setminus \partial F_0$ and prescribed boundary value $\infty$
for any approach to $\inter(F_0).$ More precisely, we prove
existence and uniqueness of the following Dirichlet problem
$(P_\infty)$:
\begin{equation*}
({P}_\infty )\begin{cases}
 \mathcal M (u)= 0\   \mathrm{ in}\ \inter(\Pc),  \\
                             u = 0 \ \mathrm{ on }\  F_j\setminus \partial F_0,\,
j=1,\ldots,n,  \\
                             u= \infty\  \mathrm{ on }\ \mathrm{int}   (F_0),  \\
                             u\in C^2\left(\mathrm{int} (\Pc)\right) \cap
                             C^0\left(\Pc \setminus  F_0\right).
            \end{cases}
\end{equation*}
\end{theorem}

\begin{proof}
 With the aid of Theorem
\ref{Sch.B} we may use the rotational Scherk hypersurfaces as
barrier. Therefore, we obtain for any $t\in \R$ a solution $v_t$
of the vertical minimal equation in $\inter (\Pc)$ which extends
continuously to $\partial \Pc \setminus \partial F_0$, taking
prescribed zero boundary value data on $\partial \Pc \setminus
F_0$ and prescribed boundary value $t$ on $\inter (F_0)$. Now
letting $t\to \infty$  as in the proof of Theorem \ref{special
Scherk}  we have that  a subsequence of the family $\{v_t\}$
converges to a solution as desired, taking into account that the
rotational Scherk hypersurfaces give a barrier at any point of
$\Pc$.

 The
uniqueness is obtained as in the proof of the monotonicity of the
family $\{v_t\}$ in  Proposition \ref{special t}, see also the
proof of Corollary \ref{vt}.
 \end{proof}

\begin{theorem}[{\bf {\em Second Scherk type hypersurface in $\hip^n\times
 \R$}}] \label{Sch2} For any $k\in \Na, \, k\geqs 2,$ there exists a family of
 polyhedron $\Pc_k$ with $2^{n-1} k$  faces and a solution $w_k$
 of the vertical minimal equation in $\inter \Pc_k$ taking
 alternatively infinite values $+\infty$ and $-\infty$ on adjacent faces of
 $\Pc_k$. Moreover, the polyhedron $\Pc_k$  can be chosen to be
 convex and can also be chosen to be non convex.
\end{theorem}

\begin{proof}

Let us fix a point $A_0$ in $\hip^n.$ Let $\{e_1,\ldots,e_n\}$ be
a positively oriented orthornormal  basis of $T_{A_0} \hip^n$. For
$k\geqs 2$ we set $u:= \sin(\pi/k) e_1 + \cos(\pi/k)e_2$. Let
$\ga_j^+,\, j=2,\ldots, n$ and $\ga_u^+$ be the oriented half
geodesics issuing from $A_0$ and tangent to $ e_2,\ldots, e_n$ and
to $u$, respectively.  Now we choose an interior point $A_1$ on
$\ga_u^+$ and an interior point $A_j$ on $\ga_j^+,\, j=2,\ldots,
n$. Therefore, $A_0,A_1,\ldots,A_n$ are independent points of
$\hip^n.$ Let $\widetilde{\Pc}$ be the polyhedron determined by
these points. The faces are denoted by $F_0,\ldots, F_n$, with the
convention that the face $F_j$ does not contain the vertex  $A_j$,
$j=0,\ldots,n$.

Let $\Pi_i$  the totally geodesic hyperplane containing the face
$F_i.$ Observe that:

\begin{enumerate}

\item $F_1$ and $F_2$ make an interior angle equal to $\pi/k.$

\item $F_j\perp F_1,$ $F_j\perp F_2, j=3,\ldots, n.$

\item \label{orth} $ F_j\perp F_k, j,k = 3,\ldots, n\, (j\not=k).$

\end{enumerate}

Therefore, the reflections in $\hip^n$ with respect to the
geodesic hyperplanes $\Pi_1$ and $\Pi_2$ leave the other geodesic
hyperplanes $\Pi_j, j=3,\ldots, n$ globally invariant.  The first
step of the construction of the polyhedron $\Pc_k$ is the
following: Doing reflection about $F_2$ we obtain another
polyhedron with faces $F_1^\ast$ (the symmetric of $F_1$ about
$F_2$), and faces $\widetilde {F_j}$ containing $F_j$, $\widetilde
{F_j}\subset\Pi_j,\, j=3,\ldots,n. $ Notice that in the process
the face $F_2$ disappears and the interior angle between the faces
$F_1$ and $F_1^*$ is $2\pi/k.$ Furthermore, the reflection of
$F_0$ about $F_2$ generates another face $F_0^1.$

 Continuing this
process doing reflections with respect to $F_1^\ast$ and so on, we
obtain a new polyhedron $\Pc^+$ with faces
$\widehat{F_j}\subset\Pi_j,\, j=3,\ldots,n,$ $\widehat{F_j}$
containing $\widetilde F_j$, and $2k$ faces issuing from  the
successive reflections of $F_0 $. Notice that both faces $F_1$ and
$F_2$ disappear at the end of the process, that is $\Pc^+$ does
not contain any face in the hyperplane $\Pi_1$ or $\Pi_2.$

Next, let us perform the reflections about $\Pi_3.$ Doing this the
face $F_3$ disappears and we get a new polyhedron with $2\cdot 2k$
faces issuing from $F_0$ and a face in each $\Pi_j, \,
j=4,\ldots,n,$ by Property (\ref{orth}). Each such face contains
$\widehat{F_j},\, j=4,\ldots,n.$ Continuing this process doing
reflections on $\Pi_4,\ldots,\Pi_n$ we finally get a polyhedron
$\Pc_k$ with $2^{n-1}\cdot k$ faces, each one issuing from $F_0.$

Now we discuss the convexity of $\Pc_k$. Let $P\subset \hi n$ be
the geodesic two-plane containing the points $A_0,A_1$ and $A_2$.
Let $\Gamma \subset P$ be the geodesic polygon obtained by the
reflection of the segment $[A_0,A_1]$  with respect to $[A_0,A_2]$
and so on. Thus $\Gamma$ is a polygon with $2k$ sides and $2k$
vertices, among them $A_1$ and $A_2$, and $A_0$ is an interior
point of $\Gamma$. Then, the polyhedron $\Pc_k$ is convex if, and
only if, the polygon $\Gamma$ is convex too. For example, if
$d(A_0,A_1)=d(A_0,A_2)$ we get that $\Gamma$ is a regular polygon
and then is convex. On the other hand, if $d(A_0,A_1)$ is much
bigger than $d(A_0,A_2)$ then $\Gamma$ is non convex.

\medskip

 Now, considering the polyhedron $\widetilde{\Pc}$ of the
beginning, with the aid of  \thmref{Sch1}, we are able to solve
 the Dirichlet problem of the minimal equation taking
$+\infty$ value data on $F_0$  and zero value data on
$F_j\setminus F_0$, $ j=1,\ldots, n.$ Using the reflection
principle on the faces, in each step of the preceding process, we
obtain at the end of the process a solution of the minimal
equation on $\inter \Pc_k,$ taking
 alternatively infinite values $+\infty$ and $-\infty$ on adjacent faces of
 $\Pc_k,$ as desired. This accomplishes  the proof of the theorem.
\end{proof}

\medskip

The following theorem are consequence of the previous results.

\medskip

\begin{theorem}[{\bf{\em Dirichlet problem for the minimal
equation in $\hip^n\times \R$  on a $C^0$  bounded convex domain
taking continuous boundary data}}]${}$\label{C.bounded.data2}

 Let $\Omega $ be a $C^0$ bounded convex  domain and let
$g: \partial\Om\rightarrow \R$ be a continuous function.

Then, $g$ admits a unique continuous extension $u:  \Omega\cup
\partial\Om\rightarrow \R$ satisfying the
vertical minimal hypersurface equation $(\ref{minequ})$ on
$\Omega$.
\end{theorem}

\begin{proof}
The proof is a consequence of the Perron process (\thmref{perron})
and the construction of barriers at any convex point of a $C^0$
domain, using  rotational Scherk hypersurfaces (\thmref{Sch.B}).
Uniqueness follows from the maximum principle.
\end{proof}

\begin{theorem}[{\bf{\em Dirichlet problem for the minimal
equation in $\hip^n\times \R$  on a $C^0$ convex domain taking
continuous finite and asymptotic boundary data}}]
\label{C.bounded.data3}${}$

Let $\Omega \subset \hip^n$ be a $C^0$  convex domain and let $g:
\partial\Om\cup \partial_\infty\Om\rightarrow \R$ be a continuous
function.

Then $g$ admits a unique continuous extension $u:  \Omega\cup
\partial\Om\cup \partial_\infty\Om\rightarrow \R$ satisfying the
vertical minimal hypersurface equation $(\ref{minequ})$ on
$\Omega$.
\end{theorem}

\begin{proof}
Notice that working in the ball model of hyperbolic space, we have
that $g$ is a continuous function on a compact set, hence $g$ is
bounded. Therefore there
 exist supersolutions and subsolutions for the
  Dirichlet problem. The
proof is a consequence of the  Perron process (\thmref{perron})
and the constructions of barriers, using the rotational Scherk
hypersurfaces (Theorem \ref{Sch.B}) at any point of $\partial\Om$,
and using  $M_1$ at any point of $\partial_\infty\Om$ (Theorem
\ref{perron}-(\ref{Item.asym.bound})).  Uniqueness follows from
the maximum principle.
\end{proof}

\section{Existence of minimal graphs over non convex admissible domains}\label{nca}

We will establish some existence of minimal graphs on certain
admissible domains and certain asymptotic boundary, in the same
way as in \cite[Theorem 5.1 and Theorem 5.2]{SE-T}. The proofs are
the same as in the two-dimensional situation, using the
$n$-dimensional catenoids and the $n$-dimensional translation
hypersurfaces $M_d$ obtained for $n\geqs 3$ in \cite{B-SE}.
Therefore we will just state the related definitions and the
theorems without proof.

\begin{definition}[{\bf {\em Admissible unbounded domains in $\hip^n$}}]\label{adm}
${}$
 {\em
 Let $\Om \subset \hi n$ be an unbounded domain.
We say that $\Om$ is an {\em admissible domain} if each connected
component $C_0$ of $\partial \Om$ satisfies the {\em Exterior
sphere of $($uniform$)$ radius $\rho$ condition}, that is,  at any
point $p\in C_0$ there exists a sphere $S_\rho$ of radius $\rho$
such that $p\in C_0\cap S_\rho$ and $\overl{\inter
 S_\rho}\cap \Om=\emptyset.$

 If $\Om$ is an unbounded admissible domain then  we denote
 by $\rho_\Om$   the supremum of the set of these $\rho.$}

\end{definition}

Let us write down a formula obtained in \cite{B-SE} that is useful
in the sequel. Let $t=\la(a, \rho),\, \rho\geq a$, be the height
function of the upper half-catenoid in $\hip^n\times \R.$ Then as
$\rho$ goes to infinity $\la(a, \rho)$ goes to $R(a)$ where $R(a)$
is given by

\begin{equation*}\label{E-cat3-16a}
R(a) := \sinh(a) \int_1^{\infty} \big( \sinh^2(a) s^2 + 1
\big)^{-1/2} \big( s^{2n-2}-1 \big)^{-1/2} \, ds.
\end{equation*}
Furthermore, the function $R$ increases from $0$ to $\pi/ (2n-2)$
when $a$ increases from $0$ to $\infty$. This means that the
catenoids in the family have finite height bounded from above by
$\pi / (n-1)$ (\cite[Proposition 3.2]{B-SE}). We set $f(\rho):=
R(\rho)$.

\begin{theorem}\label{t1}
Let $\Om\subset\hip^n$ be an admissible unbounded domain.  Let \\
$g: \partial\Om\cup\partial_\infty\Om \rightarrow \R$ be a
continuous function
 taking zero  boundary value data on $\partial\Om$.
 Let $\Gai\subset \pain \hip^n \times \R$ be
 the graph of $g$ restricted to $\partial_\infty\Om$.
 \\
If   the height function  $t$ of $\Gai$ satisfies
$-f(\rho_\Om)\leqs t\leqs f(\rho_\Om),$ then there exists a
vertical minimal graph over $\Om$ with finite boundary $\partial
\Om$ and asymptotic boundary $\Gai.$
%Particularly, if $\Gai$ is a circle we obtain a minimal graph with
%boundary $C$ asymptotic to a catenoid.

\vskip1mm

Furthermore, there is no such minimal
 graph, if  $\partial \Om$ is compact and the height function $t$  of $\Gai $
satisfies $\vert t \vert >\pi/(2n-2)$.

%Particularly, if $\Gai$ is a circle we obtain a minimal graph with
%boundary $C$ asymptotic to a catenoid.

\end{theorem}

\begin{definition}[\bf {\em E-admissible unbounded domains in
$\hip^n$}]\label{d3} ${}$

 {\em
 Let $\Om$ be an unbounded domain in
$\hip^n$ and let $\partial \Om$ be its boundary.
 We say that
$\Om$ is an {\em E-admissible domain} if
 there exists $r>0$ such that  each point of $\partial \Om$
satisfies the {\em exterior equidistant hypersurface of
$($uniform$)$ mean curvature $\tanh r$ condition}; that is,  at
any point $p\in
\partial \Om$ there exists an equidistant hypersurface  $E_{r}$ of a geodesic
hyperplane, of mean curvature $\tanh r$ (with respect to the
exterior unit normal to $\Om$ at $p$), with $p\in
\partial\Om\cap E_r$ and $ E_r\cap \Om=\emptyset.$

If $\Om$ is an unbounded E-admissible domain then  we denote
 by $r_\Om$ $\geqs 0$   the infimum  of the set of these $r.$
 If $\Om$ is a convex E-admissible domain then $r_\Om=0$.

}
\end{definition}
Thus every E-admissible domain is an admissible domain.

If $\Om$ is a convex domain  then $\Om$ is an E-admissible domain.

If each connected component $C_0$ of $\partial \Om$ is  an
equidistant hypersurface then $\Om$ is an E-admissible (maybe non
convex) domain.

Let us write down again some formulas extracted from \cite{B-SE}.
Up to a vertical translation, the height $t=\mu_{+}(a,\rho)$ of
the translation hypersurface $M_d$, $d>1$, is given by

\begin{equation*}\label{E-trans3-6}
\mu_{+}(a,\rho) = \cosh(a) \, \int_1^{\cosh(\rho)/\cosh(a)}
(s^{2n-2} - 1)^{-1/2} \, (\cosh^2(a) s^2 - 1)^{-1/2} \, ds.
\end{equation*}\bigskip

These integrals converge at $s=1$ and when $\rho \to +\infty$,
with limit value
\begin{equation*}\label{E-trans3-7}
T(a) := \cosh(a) \, \int_1^{\infty} (s^{2n-2} - 1)^{-1/2} \,
(\cosh^2(a) s^2 - 1)^{-1/2} \, ds.
\end{equation*}
$T$ is a decreasing function of $a$, which tends to infinity when
$a$ tends to zero (when $d>1$ tends to $1$) and to $\pi/(2n-2)$
when $a$ (or $d$) tends to infinity (\cite[Equations 3.55, 3.56,
3.57]{B-SE}).

We set $H(r):= T(r).$

\begin{theorem}\label{t2}

Let $\Om\subset\hip^n$ be an E-admissible unbounded domain.  Let
 $g: \partial\Om\cup\partial_\infty\Om \rightarrow \R$ be a
continuous function
 taking zero  boundary value data on $\partial\Om$.
 Let $\Gai\subset \pain \hip^n \times \R$ be
 the graph of $g$ restricted to $\partial_\infty\Om$.
 \\
If   the height function  $t$ of $\Gai$ satisfies $-H(r_\Om)\leqs
t\leqs H(r_\Om),$ then there exists a vertical minimal graph over
$\Om$ with finite boundary $\partial \Om$ and asymptotic boundary
$\Gai.$

\end{theorem}

\section{Minimal graphs in $\R^ {n+1}=\R^ n\times
\R$.}

We will write-down in this section some natural extensions of the
previous constructions to obtain minimal graphs
 in the $n+1$- Euclidean space. The proof of the related results
 for minimal graphs in $\R^ {n+1}$
 are {\em mutatis mutandis} the same as in $\hip^ n\times \R$, but simpler. So we will just
 summarize them.

  The dictionary to perform the understanding of the
 structure of the proofs is as follows: The hypersurface corresponding
 to  the family $M_d$ ($d<1$) to provide barriers at a strictly convex point
 for minimal solutions
 when the ambient space is $\hip^ n \times \R$ is the family of
 hyperplanes in $\R^ {n+1}.$ The hypersurface corresponding to
 $M_1$ to get height estimates at a compact set in the domain
 $\Omega$ is now the family of $n$-dimensional catenoids.

  The reflection principle for minimal graphs in Euclidean space can be proved in
 the same way as in Lemma \ref{ref}.  Finally
 we note that the
 Perron process  is classical in Euclidean space.

 \vskip2mm

 We now consider {\em special rotational domain} in $\R^ n$. The
 definition is analogous to Definition \ref{special domain}.
Now the curve $\ga$ is a straight line and we choose a smooth
curve $\alpha\subset P$ joining $A_0$ and $A_1$ such that the
hypersurface $\Sigma$ generated by rotating $\alpha$ with respect
to $\gamma$ has the following properties.
\begin{enumerate}
\item $\Sigma$ is smooth except possibly  at point $A_0$.
\item $\Sigma$ is strictly convex.
\item \label{transv1} $\inter (\Sigma)\setminus \{A_0\}$ is
transversal to
 the parallel lines to $\ga$.
\end{enumerate}

 We recall the minimal equation in
 $\R^ {n+1}$:

 $$
 \diver\left(\frac{\nabla u}{W(u)}\right):= \sum\limits_{i=1}^n\frac {\partial}{\partial
x_i}\left(\frac{ u_{x_i}}{\sqrt{1 +\|\nabla u\|^2_{{\R^n}}}}
\right)=0
 $$

 (just make $\la=1$ and $H=0$ in Equation (\ref{meaequ})). Explicitly, we
 have that the minimal equation in $\R^ {n+1}$ is given by
$$
\sum\limits_{i=1}^n\left(1 +(u_{x_1}^2
+\cdots+\wh{u_{x_i}^2}+\cdots +u_{x_n}^2)\right)u_{x_i x_i} -
2\sum\limits_{i<k}u_{x_i}u_{x_k}u_{x_i x_k}=0
$$

\begin{theorem}[\bf{\em Rotational Scherk hypersurface}] \label{special ScherkE}
 Let $\Dc_\Sigma\subset \R^ n$ be a
special rotational domain. There is a unique solution $v$ of the
vertical minimal equation in $\inter (\Dc_\Sigma)$ which extends
continuously to   $\inter (\Sigma)$, taking prescribed zero
boundary value  and taking prescribed boundary value $\infty$ for
any approach to $U_\Sigma$.

More precisely,  the following Dirichlet problem admits a unique
solution $v$.
\begin{equation*}
\begin{cases}
 \sum\limits_{i=1}^n\frac {\partial}{\partial
x_i}\left(\frac{ u_{x_i}}{\sqrt{1 +\|\nabla u\|^2_{{\R^n}}}} \right)= 0 \  \mathrm{ on }\  \inter (\Dc_\Sigma), \\
                             u = 0 \   \mathrm{ on }\  \inter (\Sigma), \\
                             u= +\infty \ \mathrm{ on }\  U_\Sigma,\\
                             u\in C^2\left(\inter (\Dc_\Sigma)\right) \cap
                      C^0\left(\Dc_\Sigma \setminus \ov{U}_\Sigma \right).
            \end{cases}
\end{equation*}
We call the graph of $v$ in $\R^ {n+1}$ a {\em rotational Scherk
hypersurface}.

\end{theorem}

\begin{proof}

We first solve the auxiliary Dirichlet problem ($P_t$) taking zero
boundary value data on the interior of $\Sigma$ and prescribed
boundary value $t$ on $U_\Sigma$, in the same way as in the
Proposition \ref{special t}. On account that the family of
$n$-dimensional catenoids provides an upper and lower barrier to a
solution over any compact set of $\inter (\Dc_\Sigma),$  letting
$t\goto \infty$ we get the desired solution.

 Uniqueness is shown
in the same way as the proof of monotonicity in Proposition
\ref{special t}.
\end{proof}

We observe that this result was also obtained by A. Coutant
\cite{Co} using a different approach.

\begin{theorem}[{\bf{\em Barrier at a $C^0$ convex
point}}]\label{Sch.BE} Let $\Omega\subset \R^ n$ be a domain and
let $p_0\in \partial \Omega$ be a boundary point where $\Omega$ is
$C^0$  convex. Then for any bounded data $g:\partial \Omega
\rightarrow \R$ continuous at $p_0$ the family of rotational
Scherk hypersurfaces  provides a barrier at $p_0$.

\end{theorem}

\begin{proof}
 The proof is the same, but simpler, as the proof of Theorem
\ref{Sch.B}. More precisely the proofs of the analogous of Claim 1
and  2 are simpler, passing first by the solution $v_t$ of the
related auxiliary Dirichlet problem $(P_t)$.
\end{proof}

\begin{corollary}[{\bf{\em Rotational Scherk hypersurface}}]\label{coneE}
 Let $\Dc_\Sigma\subset \R^ n$ be a special rotational domain
generated by a segment $\a$ of a straight line. Then:
\begin{enumerate}
\item \label{item ScherkE}  There is a unique solution $v$
of the vertical minimal equation in $\inter (\Dc_\Sigma)$ which
extends continuously to $\inter (\Sigma) \cup U_\Sigma$, taking
prescribed zero boundary value data on the interior of $\Sigma$
and prescribed boundary value $\infty$ on $U_\Sigma$.

We also call the graph of $v$ in $\R^ {n+1}$ a {\em rotational
Scherk hypersurface}.

\vskip2mm

\item Let $\Omega\subset \R^ n$ be a domain and
let $p_0\in \partial \Omega$ be a boundary point where $\Omega$ is
$C^0$  convex. Then for any bounded data $g:\partial \Omega
\rightarrow \R$ continuous at $p_0$ the family of rotational
Scherk hypersurfaces  given in the first statement provides a
barrier at $p_0$.
\end{enumerate}
\end{corollary}

We define the notion of {\em admissible polyhedron } in $\R^ n$ in
the same way as in hyperbolic space, see Definition \ref{pol}. The
following result is proved in the same way as in Theorem
\ref{Sch1}.

\begin{theorem}[{\bf {\em First Scherk type hypersurface in $ \R^ {n+1}$}}]\label{Sch1E}
  Let $\mathcal P$ be an admissible convex polyhedron in $\R^ n$. There
  is a unique solution $v_\infty$ of the vertical minimal
equation in $\inter (\Pc)$ extending continuously  to $\partial
\Pc\setminus  F_0$, taking prescribed zero boundary value data on
$F_1\setminus \partial F_0,\ldots ,F_n\setminus \partial F_0$ and
prescribed boundary value $+\infty$   for any approach to
$\inter(F_0).$ More precisely, we prove existence and uniqueness
of the following Dirichlet problem $(P_\infty)$:

\begin{equation*}
({P}_\infty )\begin{cases} \sum\limits_{i=1}^n\frac
{\partial}{\partial x_i}\left(\frac{ u_{x_i}}{\sqrt{1 +\|\nabla
u\|^2_{{\R^n}}}} \right)= 0\   \mathrm{ on}\ \inter(\Pc),  \\
                             u = 0\  \mathrm{ on }\   F_j\setminus \partial F_0,\,
j=1,\ldots,n,  \\
                             u= +\infty \  \mathrm{ on }\ \mathrm{int}   (F_0),  \\
                             u\in C^2\left(\mathrm{int} (\Pc)\right) \cap
                             C^0\left(\Pc \setminus  F_0\right).
            \end{cases}
\end{equation*}
\end{theorem}

We remark that the above result is also obtained  by A. Coutant
\cite{Co}.

\smallskip

Next theorem can be proved exactly as in Theorem \ref{Sch2}.

\begin{theorem}[{\bf {\em Second Scherk type hypersurface in $
 \R^{n+1} $}}] \label{Sch2E} For any $k\in \Na, \, k\geqs 2,$ there exists a family of
 polyhedron $\Pc_k$ with $2^{n-1} k$  faces and a solution $w_k$
 of the vertical minimal equation in $\inter \Pc_k$ taking
 alternatively infinite values $+\infty$ and $-\infty$ on adjacent faces of
 $\Pc_k$.  Moreover, the polyhedron $\Pc_k$  can be chosen to be
 convex and can also be chosen to be non convex.
\end{theorem}

\begin{remark}
{\em When the ambient space is $\R^ 4$  with the aid of Theorem
\ref{Sch2E} we have a   solution  of the minimal equation in the
interior of an octahedron  in $\R^ 3$ taking alternatively
infinite values $+\infty$ and $-\infty$ on adjacent faces. Indeed,
using the notations of the proof of Theorem \ref{Sch2}, we set
$k=2$ and we choose $A_1,A_2$ and $A_3$ so that
$d(A_1,A_2)=d(A_1,A_3)=d(A_2,A_3)$. Thus the polyhedron $\Pc_2$
obtained is an octahedron. }

\end{remark}

\end{document}